\documentclass[11pt]{amsart}
\usepackage{amssymb, hyperref, enumerate}

\numberwithin{equation}{section}

\theoremstyle{plain}
\newtheorem{theorem}{Theorem}[section]
\newtheorem{lem}[theorem]{Lemma}
\newtheorem{prop}[theorem]{Proposition}
\newtheorem{cor}[theorem]{Corollary}
\theoremstyle{definition}

\usepackage{mathtools}

\newcommand{\inn}[1]{ \langle #1 \rangle }
\newcommand{\R}{\mathbb{R}}
\newcommand{\N}{\mathbb{N}}

\newcommand{\C}{\mathbb{C}}

\newcommand{\op}{_{L^2 \to L^2}}
\newcommand{\opn}[1]{_{L^2(\R^{{#1}}) \to L^2(\R^{{#1}})}}

\newcommand{\wh}{\widehat}
\newcommand{\om}{{\Omega}}
\newcommand{\cO}{{\mathcal{O}}}

\newcommand{\les}{\lesssim}

\newcommand{\norm}[1]{ || #1 ||}

\newcommand{\ind}{\mathbf{1}}
\newcommand{\sgn}{\operatorname{sgn}}

\AtBeginDocument{%
   \def\MR#1{}
}

\begin{document}
\title{$L^2$ Bounds for a maximal directional Hilbert transform}

\author{Jongchon Kim}
\address{Department of Mathematics, University of British Columbia, 1984 Mathematics Road, Vancouver, BC, Canada V6T 1Z2}
\email{jongchon.kim.work@gmail.com}

\author{Malabika Pramanik}
\address{Department of Mathematics, University of British Columbia, 1984 Mathematics Road, Vancouver, BC, Canada V6T 1Z2}
\email{malabika@math.ubc.ca}

\begin{abstract}
Given any finite direction set $\Omega$ of cardinality $N$ in Euclidean space, we consider the maximal directional Hilbert transform $H_{\Omega}$ associated to this direction set. Our main result provides an essentially sharp  uniform bound, depending only on $N$, for the $L^2$ operator norm of $H_{\Omega}$ in dimensions 3 and higher. The main ingredients of the proof consist of polynomial partitioning tools from incidence geometry and an almost-orthogonality principle for $H_{\Omega}$. The latter principle can also be used to analyze special direction sets $\Omega$, and derive sharp $L^2$ estimates for the corresponding operator $H_{\Omega}$ that are typically stronger than the uniform $L^2$ bound mentioned above. A number of such examples are discussed. 
\end{abstract}
\date{\today}
\maketitle
{\allowdisplaybreaks 
\section{Introduction}
\noindent Operators associated with sets of directions form a central theme in harmonic analysis. They arise, for instance, in the study of differentiation of integrals, in multiplier problems for the polygon and in Bochner-Riesz means \cite{CorK, CorP, CF, NSW, CorB}. The present article is concerned with a specific directional operator in this class, namely the maximal directional Hilbert transform. 
\vskip0.1in
\noindent Given a nonzero vector $\vec{\omega} \in \mathbb R^{n+1}$, the {\em{directional Hilbert transform}} on $\mathbb R^{n+1}$ {\em{in the direction of $\vec{\omega}$}}  is the operator that maps 
\begin{equation} \label{DHT-single} 
f \longmapsto \text{p.v.} \frac{1}{\pi} \int_{\mathbb R} f(x- t \vec{\omega}) \frac{dt}{t}, \qquad x \in \mathbb R^{n+1}, 
\end{equation} 
where the integral is interpreted in the principal value sense. The transform remains invariant if $\vec{\omega}$ is replaced by any nonzero scalar multiple of it. Without loss of generality and after a permutation of coordinates if necessary, we will think of $\vec{\omega}$ as a vector of the form $\vec{\omega} = \langle v, 1 \rangle$, with $v \in \mathbb R^{n}$. The corresponding operator \eqref{DHT-single} will be denoted by $H_{v}$. 
\vskip0.1in
\noindent Let $\om$ be a set of points in $\R^n$. 
The {\em{maximal directional Hilbert transform}} $H_\om$ associated with the set of directions $\{\langle v,1 \rangle : v \in \Omega  \} \subseteq \mathbb R^{n+1}$ is defined as follows:
\begin{equation} \label{def-MDHT} H_\om f(x) := \sup_{v\in \om} |H_v f(x)|, \qquad x \in \mathbb R^{n+1}.\end{equation}  
By a slight abuse of nomenclature, we will refer to $\Omega$ as the {\em{direction set}} underlying the maximal operator $H_{\Omega}$. For $1 < p < \infty$, it follows from well-known properties of the classical univariate Hilbert transform that for any single vector $v \in \mathbb R^n$, the operator $H_v$ is bounded on $L^p(\R^{n+1})$, with the operator norm uniform in $v$. From this, one concludes easily that $H_\om$ is bounded on $L^p(\R^{n+1})$ if $\om$ is finite. Remarkably, the converse is also true. A collective body of work, pioneered by Karagulyan \cite{Kar} and extended by \L aba, Marinelli and the second author \cite{LMP}, shows that for every $n\geq 1$ and every $1 < p < \infty$, there exists an absolute constant $c = c(p, n) > 0$ such that the operator bound
\begin{equation} \label{maxDHT-lowerbound}  ||H_{\Omega}||_{p \rightarrow p} \geq c \sqrt{\log N} \end{equation}  
holds for {\em{every}} finite direction set $\Omega$ of cardinality $N$. Here $|| H_{\Omega} ||_{p \rightarrow p}$ represents the operator norm of $H_{\Omega}$ from $L^p(\mathbb R^{n+1})$ to itself. The lower bound in \eqref{maxDHT-lowerbound} goes to infinity as $N \rightarrow \infty$, regardless of the structure of $\Omega$. 
\vskip0.1in
\noindent On the other hand, the behaviour of the same operator $H_{\Omega}$ is very different when applied to functions that are localized to a single frequency scale. Lacey and Li \cite{{Lacey-Li-1},{Lacey-Li-2}} have shown that the operator $f \longmapsto H_{\mathbb S^1} (\zeta \ast f)$ maps $L^2(\mathbb R^2)$ to weak $L^2(\mathbb R^2)$, and $L^p(\mathbb R^2)$ to itself for $p > 2$. Here $\zeta$ is a Schwartz function in $\mathbb R^2$ with frequency support in the annulus $\{1 \leq |\xi| \leq 2\}$. The unboundedness phenomenon displayed by $H_{\Omega}$ for infinite $\Omega$ is also in sharp contrast with the behaviour of another closely related operator, the directional maximal function $M_{\om}$, which is known to be $L^p$-bounded for certain infinite direction sets \cite{{Alf}, {Carbery88}, {NSW}, {PR}, {Sjogren-Sjolin}}. Let us recall that for any set $\Omega \subseteq \mathbb R^n$ that could be finite or infinite, 
\begin{align}  M_\om f (x) &:= \sup_{v\in \om} M_v f(x),   \label{max-dir} \text{ where } \\ M_v f(x) &:= \sup_{h>0}\frac{1}{2h}\int_{-h}^h |f(x-\vec{\omega}t)| dt \quad \text{ for }   \vec{\omega} = \langle v,1 \rangle. \nonumber \end{align} 
\vskip0.1in
\noindent The distinctive features of $H_{\Omega}$ have led to several questions of interest. For example, 
\begin{itemize} 
\item {\em{Question 1: }} What is a uniform, and in general sharp, upper bound on the $L^p(\mathbb R^{n+1})$ operator norm of $H_{\Omega}$ that depends only on $\#(\Omega) = N$? 
\vskip0.1in 
\item {\em{Question 2: }} Given a fixed cardinality $N$, under what additional geometric assumptions on $\Omega$ can the above uniform bound be improved? 
\end{itemize} 
These two questions are the primary focus of this article. 
\subsection{Main results} 
\subsubsection{General direction sets in $\mathbb R^n$, $n \geq 2$}
As we will see in section \ref{Literature-review-section} below, question 1 is relatively well-studied in $\mathbb R^2$, i.e., when $n=1$, but is less understood in higher dimensions. 
Our first main result addresses question 1 for $n\geq 2$ and $p=2$. Throughout the paper, we allow all implicit constants to depend on $n$.
\begin{theorem} \label{thm:highd} Let $n\geq 2$. Then for any $\epsilon > 0$, there exists a constant $C_{\epsilon}>0$ depending only on $n$ and $\epsilon$ such that for any finite direction set $\om \subset \R^n$ of cardinality $N$, the following estimate holds: 
\begin{equation}  \norm{ H_\om f}_{L^2(\R^{n+1})} \leq C_\epsilon N^{\frac{n-1}{2n} +\epsilon} \norm{f}_{L^2(\R^{n+1})}. \label{L2-main-bound} \end{equation} 
\end{theorem}
\noindent {\em{Remarks: }} 
\begin{enumerate}[1.] 
\item The bound \eqref{L2-main-bound} is sharp, except possibly the factor of $N^\epsilon$. This follows from the work of Joonil Kim \cite{Kim}, who proves the following lower bound when $\Omega$ is the $n$-fold Cartesian product of a uniform direction set: there exists a constant $c > 0$ such that
\begin{equation} \label{def-uniform}
\begin{split}
||H_{\Omega}||_{2 \rightarrow 2} &\geq c N^{\frac{n-1}{2n}}  \quad \text{ for } \Omega = U_{M}^n, \; N = M^n, \text{ where }  \\ 
U_M &= \bigl\{j/M : 1 \leq j \leq M \bigr\}.  
\end{split}
\end{equation}   
\vskip0.1in
\item \label{improvement-3} For $n=2$, i.e., in dimension 3, we are able to improve upon the estimate \eqref{L2-main-bound} by replacing $N^{\epsilon}$ with a slowly increasing function $h(N)$ that goes to infinity as $N \rightarrow \infty$. The implicit constant $C_{\epsilon}$ in \eqref{L2-main-bound} is then replaced by a constant that depends on $h$. The precise statement of this may be found in Theorem \ref{thm:3dgen} in section \ref{Example-section}. In particular, $h$ can be chosen to be the $k$-fold logarithm function for any $k \geq 1$, i.e., $h_k(N) = \log_k N = \log (1 + h_{k-1}(N))$, with $h_0(N) = N$. \label{remark2-thm:highd} 
\end{enumerate} 
\vskip0.1in
\subsubsection{Direction sets in algebraic varieties} We obtain Theorem \ref{thm:highd} as the consequence of a more general result that involves direction sets contained in algebraic varieties. An (affine) \emph{algebraic variety} in $\mathbb C^n$ is the common zero set of finitely many polynomials in $n$ complex variables. In section \ref{sec:poly}, we provide definitions of the dimension and degree of a variety, as well as the relevant facts needed for this article. 
Our main result, Theorem \ref{thm:higherdim} below, provides a uniform bound on the $L^2(\R^{n+1})$-operator norm of $H_\om$ when $\Omega$ is any finite subset of an algebraic variety of prescribed dimension and degree. This partially addresses question 2. 
\vskip0.1in
\noindent Let $\mathcal V(m, n, d)$ denote the collection of all algebraic varieties $V$ in $\mathbb C^n$ of dimension at most $m$ and degree at most $d$. Set 
\[ V(\mathbb R) := V \cap \bigl[ \mathbb R^n + i \{\vec{0}\} \bigr] = \{ x \in \mathbb R^n: x \in V\}. \]  In other words, $V(\mathbb R)$ is the purely real subset of $V$. 
\vskip0.1in 
\begin{theorem}\label{thm:higherdim}
Let $d\in \N$ and $m,n$ be integers such that $n\geq 2$ and $0\leq m \leq n$. For every $\epsilon >0$, there are constants $A_{\epsilon}(m, d) > 0$ such that for any $V \in \mathcal V(m, n, d)$ and any finite direction set $\om  \subset V(\R)$ of cardinality $N$, the following estimates hold:
\begin{equation*}
 ||H_{\Omega}||_{2 \rightarrow 2} \leq 
\begin{cases}
  d &\text{ when } m=0, \\ 
 A_{\epsilon}(m, d) N^{\frac{m-1}{2m} +\epsilon} &\text{ when } 1\leq m \leq n.
\end{cases}
\end{equation*}
\end{theorem}

\noindent {\em{Remarks: }} 
\begin{enumerate}[1.]
\item \label{higherdim-implies-highd} The $n$-dimensional complex Euclidean space $\mathbb C^n$ is itself a variety of dimension $n$ and degree 1. Thus Theorem \ref{thm:highd} is a special case of Theorem \ref{thm:higherdim} with $m=n$ and $d=1$ for $V = \mathbb C^n$. 
\vskip0.1in
\item The given bound is trivial for $m=0$; it is a consequence of the fact that the degree of a zero dimensional variety $V$ coincides with its cardinality. 
\vskip0.1in
\item In contrast with the definition in section \ref{sec:poly}, certain texts (see for example \cite{Sheffer}) define the degree of an algebraic variety $V$ as the smallest integer $D$ such that $V$ can be represented as the common zero set of finitely many complex polynomials of degree at most $D$. While these two notions are not identical, each controls the other, as shown in Lemmas 4.2 and 4.3 of \cite{ST}. Let us define $\mathcal V'(m, n, D)$  as the class of all $m$-dimensional varieties in $\C^n$ that can be wriitten as the common zero set of finitely many polynomials of degree at most $D$. Then a statement similar to Theorem \ref{thm:higherdim} remains valid with $\mathcal V(m,n,d)$ replaced by $\mathcal V'(m,n,D)$. A small modification is necessary for $m=0$, where $d$ is replaced by $D^n$.

\vskip0.1in
\item The estimate in Theorem \ref{thm:higherdim} does not quantify the dependence on $m$ and $d$, but is sharp in $N$ for every $1 \leq m \leq n$, except possibly the factor of $N^\epsilon$. This follows by choosing $V = \mathbb C^m \times \{0\}$, for which $V(\mathbb R) = \mathbb R^m \times \{0\} \subseteq \mathbb R^n$, and setting $\Omega = U_M^m \times \{0\}$, with $N = M^m$ and $U_M$ as in \eqref{def-uniform}. It then follows from a standard slicing argument (see Lemma \ref{lem:slice} in the appendix) that $||H_{\Omega}||_{2 \rightarrow 2} \geq c N^{(m-1)/(2m)}$. It would be of interest to eliminate the factor $N^{\epsilon}$ and to quantify the dependence of the implicit constant on the degree $d$. We make partial progress on this issue for $n=2$, as shown in Theorem \ref{thm:3d}. This leads to the improvement of Theorem \ref{thm:highd} in dimension 3 discussed earlier (in remark \ref{improvement-3} following Theorem \ref{thm:highd}).
\end{enumerate}

\subsubsection{An almost orthogonality principle} A crucial ingredient of Theorem \ref{thm:higherdim} is an almost-orthogonality principle for $H_\om$, which may be of independent interest. Indeed all the new results in this paper (including those in sections \ref{Example-section} and \ref{Example-section-2}) depend on it. We state the result below after setting up the relevant notation. 
\vskip0.1in
\noindent Let $\mathbb{O} =\{ O_j \}$ be a finite collection of non-empty sets in $\mathbb R^n$, often called ``cells". For each unit vector $u \in \mathbb S^{n} \subseteq \mathbb R^{n+1}$, we define $E_{\mathbb{O}}(u)$ to be the number of cells $O_j \in \mathbb{O}$ that intersect the hyperplane $Z(P_u)=\{y\in \R^{n} : P_u(y)=0\}$, where 
\begin{equation} \label{def-Pu}
P_u(y) := u \cdot \inn{y,1}.
\end{equation} 

\begin{theorem} \label{thm:ortho}
Let $\om$ be any finite set in $\R^n$, $n \geq 1$. Given a finite collection of non-empty connected sets
$\mathbb{O}=\{ O_j \}$ in $\mathbb R^n$ covering $\om$, we set \[\om_j := \om \cap O_j, \quad \text{ so that }  \quad \om = \bigcup_j \om_j .\] For each $j$, we fix an element $v_j\in O_j$ and denote by $\cO$ the collection of chosen points $v_j$. Then with $E(u)=E_{\mathbb{O}}(u)$, the following estimate holds: 
\begin{equation} \label{almost-ortho} 
\norm{H_\om}_{2 \rightarrow 2} \leq \norm{H_\cO}_{2 \rightarrow 2} + \norm{E}_{L^\infty(\mathbb S^{n})}^{1/2}  \left(\max_j \norm{H_{\om_j}}_{2 \rightarrow 2}+1 \right). \end{equation} 
\end{theorem}
\vskip0.1in
\noindent {\em{Remarks: }} 
\begin{enumerate}[1.]
\item The assumption that the set $O_j$ is connected is used in the proof only in the following way; for every $u\in \mathbb S^{n}$, if $P_u(x) \neq 0$ for every $x\in O_j$, then either $P_u(x) >0 $ for every $x\in O_j$ or $P_u(x) < 0 $ for every $x\in O_j$. 

\vskip0.1in
\item Various versions of almost orthogonality have been used to study $H_\om$, although not in the generality of Theorem \ref{thm:ortho}. In particular, Theorem \ref{thm:ortho} is inspired by the work of Joonil Kim \cite{Kim}, where he uses an inductive argument based on the Fourier localization of the difference $H_v f - H_{v'}f$ to obtain sharp bounds on $\norm{H_\om}_{2 \rightarrow 2}$ for direction sets $\om$ given by Cartesian products; see \eqref{eqn:kimprod} below. See also \cite[Theorem 5.1]{DP2} for a version of the almost orthogonality principle with a fixed choice of the cells $\{ O_j\}$ in $\R^3$.

\vskip0.1in
\item \label{remark-n1} The statement of Theorem \ref{thm:ortho} is particularly simple when $n=1$. In this case, the cells $\{O_j\}$ can be chosen as disjoint intervals covering $\om\subset \R$, and $Z(P_u)$ contains at most a single point, so that $\norm{E}_{L^\infty(\mathbb S^{1})} \leq 1$ trivially. Thus for $n=1$, we have 
\begin{equation}  \norm{H_\om}\op \leq \norm{H_\cO}\op + \max_j \norm{H_{\om_j}}\op+1. \label{almost-ortho-n1} \end{equation}  
\vskip0.1in
\item Almost orthogonality estimates similar to \eqref{almost-ortho-n1} have historically played an important role in obtaining bounds for other directional maximal operators, such as $M_{\Omega}$ defined in \eqref{max-dir}; see \cite{{Alf},{ASV2},{ASV}, {PR}}. For example, in \cite{{ASV2}, {ASV}}, the authors derive an almost orthogonality principle for $M_{\Omega}$ in $L^2$, and use it to give a simple proof of the estimate 
\[ \norm{M_\om}_{L^2(\R^2) \to L^2(\R^2)} \leq C \log N \quad \text{ for any $\Omega \subset \mathbb R$ with $\#(\Omega) = N$,} \]
originally due to Katz \cite{Katz}. In \cite{Alf}, Alfonseca proves yet another orthogonality principle for $M_{\Omega}$ in $L^p$ that can be applied in a variety of contexts. In particular, it is used to reprove $L^p(\mathbb R^2)$ bounds, originally shown by Sj\"ogren and Sj\"olin \cite{Sjogren-Sjolin}, for $M_{\Omega}$ where $\Omega$ is a (possibly infinite) lacunary set of finite order. A similar $L^p(\mathbb R^n)$ orthogonality estimate for $n \geq 2$ appears in \cite[Theorem A]{PR}.   

\vskip0.1in
\item Theorem \ref{thm:ortho} permits a range of applications. In addition to proving Theorem \ref{thm:higherdim}, it provides simpler proofs for certain known bounds on $H_{\Omega}$, in some cases with small improvements. A few such applications have been discussed in Section \ref{Example-section}. More interestingly, Theorem \ref{thm:ortho} can be used to obtain new and sharp bounds on $H_{\Omega}$ that are stronger than the general bound \eqref{L2-main-bound}, for direction sets $\Omega$ with special algebraic or geometric properties. This turns out to be the case, for example, when $\Omega$ is given by points on an algebraic variety as in Theorem \ref{thm:higherdim}, or if $\Omega$ is of product type; see Theorem \ref{thm:prod}. A number of such applications have been discussed in section \ref{Example-section-2}. 
\end{enumerate} 


\subsection{Literature review} \label{Literature-review-section} We give a brief survey of some earlier results to place ours in context. In $\mathbb R^2$, i.e., for the case $n=1$, it is known that there exists an absolute constant $C > 0$ such that 
\begin{equation}\label{eqn:2d}
\norm{H_\om}_{L^2(\R^2) \to L^2(\R^2)} \leq C \log N
\end{equation}
for any direction set $\om$ of cardinality $N$. This estimate can be traced back to the work of Christ, Duoandikoetxea and Rubio de Francia; it follows, for example, from their paper \cite[Theorem 2]{CDR}, by setting $n = 2$ and $\Gamma = \mathbb S^1$. Alternative proofs may be found in \cite{Kar, Kim}. The bound in \eqref{eqn:2d} is optimal and is attained for the uniform direction set $U_N$ given by \eqref{def-uniform}, see \cite{Kim}.  
The estimate \eqref{eqn:2d} was extended to maximal directional singular integrals in \cite{Dem0} and to $L^p$ estimates for $p>2$ in \cite{DD}. 
\vskip0.1in
\noindent We turn now to special direction sets $\Omega$. For lacunary direction sets such as $ \om = \{ 2^{-k} : 1\leq k\leq N \} \subseteq \mathbb R$, it is known that 
\begin{equation}\label{eqn:lac}
c \sqrt{\log N} \leq  \norm{H_\om}_{L^p(\R^2) \to L^p(\R^2)} \leq C \sqrt{\log N}
\end{equation}
for all $1<p<\infty$. The upper bound in \eqref{eqn:lac} is due to Demeter and Di Plinio \cite{DD}. See also \cite{Dem, DD, DP} for generalizations of these results to directional singular integral operators and to finite order lacunary directions, respectively. As mentioned earlier in \eqref{maxDHT-lowerbound}, the lower bound in \eqref{eqn:lac} has been shown to hold for any direction set $\Omega$ in $\mathbb R^n$ with $N$ elements \cite{{Kar}, {LMP}}. 
\vskip0.1in
\noindent In dimensions $n\geq 2$, the bound
\[ \norm{H_\om}_{L^2(\R^{n+1}) \to L^2(\R^{n+1})} \leq C D\log N \]
was obtained in \cite[Theorem 2]{CDR} with an absolute positive constant $C$ for direction sets $\om$ contained in a curve in $\R^n$ which crosses every hyperplane at most $D$ times. A set $\Omega$ of this form is a subset of one-parameter family of directions, with the single parameter ranging over the curve. In contrast, Joonil Kim \cite{Kim} considers direction sets that may be viewed as genuinely ``$n$-dimensional''. For direction sets given by Cartesian products $\om=\om_1 \times \cdots \times \om_n$, with $\om_j \subset \R$ and $\#(\Omega_j) = N_1$ for all $j$,  \cite{Kim} establishes the following estimate:  
\begin{equation}\label{eqn:kimprod}
\norm{H_\om}\opn{n+1} \leq C {N_1}^{\frac{n-1}{2}} = C N^{\frac{n-1}{2n}}, 
\end{equation}
where $\#(\Omega) = N = N_1^n$. The article \cite{Kim} also shows that the bound \eqref{eqn:kimprod} is sharp for a specific member of this class, namely $\Omega = U_{N_1}^n$. Here $U_{N_1}$ refers to the uniform direction set defined in \eqref{def-uniform}. Incidentally, these direction sets of product type offer examples in support of the sharpness of \eqref{L2-main-bound}, as alluded to after the statement of Theorem \ref{thm:highd}. See also \cite{DP2, ADP} for sharp estimates of $H_{\Omega}$ in $\mathbb R^n$ for direction sets $\Omega$ that are ``finite order lacunary''.
\vskip0.1in
\noindent Recently, other geometric variants of the maximal functions $M_\om$ and $H_\om$ have been considered. For example, the articles \cite{GRSY1, GRSY2} provide $L^p$ estimates for maximal functions associated with families of homogeneous curves in $\R^2$.

\subsection{Overview of the proof} There are two main ingredients in the proof of Theorem \ref{thm:higherdim}. The first is the almost-orthogonality principle for $H_\om$, namely Theorem \ref{thm:ortho} mentioned previously, which we obtain using the square function argument from \cite{Kim}. 
The second main ingredient is polynomial partitioning, introduced by Guth and Katz \cite{GK}; see Theorem \ref{thm:GK}. 
We refer the interested reader to \cite{Guth} for a treatise on the subject, and also to the seminal papers \cite{Guth1, Guth2} for applications of polynomial partitioning to the Fourier restriction problem. 
\vskip0.1in
\noindent We briefly sketch the proof of Theorem \ref{thm:highd}, which is 
Theorem \ref{thm:higherdim} for $V = \C^n$. In this setting, the direction set $\Omega\subset \R^n$ is finite, but otherwise entirely arbitrary. 
In the absence of any structural assumptions on $\Omega$ and with the goal of applying Theorem \ref{thm:ortho}, we  choose the sets $O_j$ as the connected components of $\mathbb R^n \setminus Z(P)$, where $P$ is a partitioning polynomial. This splits the argument into two parts. The contribution from $\Omega \setminus Z(P) = \cup_j (\Omega \cap O_j)$
 admits an inductive treatment based on cardinality, since each set $\Omega \cap O_j$ contains fewer elements of $\Omega$. The contribution from $\Omega \cap Z(P)$ is treated differently. This is a subset of the zero set of the partitioning polynomial, and hence has additional structural properties; for instance, as an algebraic variety, $Z(P)$ is of dimension strictly lower than the ambient dimension $n$. To study $\Omega \cap Z(P)$, we appeal to more sophisticated polynomial partitioning for finite subsets of algebraic varieties, in particular, Theorem \ref{thm:polyvar} due to Matou\v{s}ek and Pat\'{a}kov\'{a} \cite{MP}. This opens up an inductive strategy for handling $\Omega \cap Z(P)$, based on the dimension of the ambient algebraic variety (in this case $Z(P)$). 
This approach leads naturally to the consideration of direction sets contained in algebraic varieties of a given dimension, and explains the need for Theorem \ref{thm:higherdim}. 



\vskip0.1in 
\noindent Besides the papers \cite{ASV2,Kim} discussed earlier, our work was inspired by the recent results of Di Plinio and Parissis \cite{DPavg}, where sharp $L^2$-estimates were obtained for a maximal directional averaging operator 
using polynomial methods. Interestingly, in \cite{DPavg} the authors develop and use their own variant of polynomial partitioning adapted to the problem. It turns out that, for the study of $H_\om$, it is sufficient to use polynomial partitioning tools available in the literature, specifically in \cite{GK, MP, BB, ST}. Some additional technical difficulties which exist in \cite{DPavg} have been avoided in this paper due to the availability of Theorem \ref{thm:ortho}. This theorem is based on the strong Fourier localization of the difference $H_v f - H_{v'}f$ (see Lemma \ref{lem:supp}).  We are not aware of an analogous result that exists in general dimensions for directional maximal functions. 
\subsection{Layout of the paper} In addition to Theorems \ref{thm:highd} and \ref{thm:ortho} stated in this introduction, this paper contains a number of new results pertaining to special direction sets $\Omega$. Most of them have been relegated to sections \ref{Example-section} and \ref{Example-section-2}. We take this opportunity to highlight their content and location, and describe the general organization of this paper. 
\vskip0.1in
\noindent  In section \ref{Example-section}, and as a warm-up for the main theorems, we discuss a number of applications of Theorem \ref{thm:ortho} that lead to new proofs of existing results. Section \ref{Example-section-2} is devoted to more nontrivial applications, where we obtain sharp estimates on $||H_{\Omega}||_{2 \rightarrow 2}$ for certain direction sets $\Omega$. In particular, we consider general product sets (Theorem \ref{thm:prod}) which lead to an extension of \eqref{eqn:kimprod}, and direction sets in $\mathbb R^2$ contained in the zero set of a bivariate polynomial (Theorem \ref{thm:3d}). As an application of the former and given any prescribed growth rate, we construct direction sets $\Omega$ for which $||H_{\Omega}||_{2 \rightarrow 2}$ obeys that growth rate; see Theorems \ref{thm:growth-alpha} and Corollary \ref{prescribed-growth-cor}. 
This section also contains Theorem \ref{thm:3dgen}, a refined version of \eqref{L2-main-bound} in $\mathbb R^3$ that was mentioned in remark \ref{remark2-thm:highd} following Theorem \ref{thm:highd}. 
\vskip0.1in
\noindent The remainder of the paper is devoted to proofs. In section \ref{sec:ortho}, we prove the almost orthogonality principle Theorem \ref{thm:ortho}, which is a key ingredient in all the other proofs in this paper. The subsequent sections are given over to proving the applications stated in section \ref{Example-section-2}. For instance, in sections \ref{sec:prod} and  \ref{proof:thm3d:section}, we prove Theorem \ref{thm:prod} and Theorem \ref{thm:3d}, respectively. This in turn leads to the proof of Theorem \ref{thm:3dgen}, which appears in section \ref{sec:3d}.  Polynomial partitioning tools needed for the proof of our main result, Theorem \ref{thm:higherdim}, are gathered in section \ref{sec:poly}. 
The proof of the theorem itself has been executed in Section
 \ref{section:proof:higherdim}. Appendix \ref{sec:appendix} contains a few auxiliary lemmas needed in various sections.

\subsection{Acknowledgements} This work was completed while the first author was a joint postdoctoral fellow at the Pacific Institute of Mathematical Sciences and the department of mathematics at the University of British Columbia. He would like to thank Joshua Zahl for pointing out the references \cite{MP, Mil, Za}. The second author thanks the Peter Wall Institute of Advanced Studies for its support in the form of a 2018-2019 Wall Scholarship that facilitated the project. Both authors were partially supported by a Discovery grant from the Natural Sciences and Engineering Research Council of Canada.    

\section{Examples and applications: Part 1} \label{Example-section} 
\noindent As mentioned in the introduction, Theorem \ref{thm:ortho} can be applied directly to certain direction sets $\Omega$ that have been studied in the literature, to yield new proofs of existing results concerning $H_{\Omega}$, in some cases with optimal bounds. This section is given over to a discussion of such applications, as preparation for the core ideas that appear in more refined form in the proofs of our main results.  
\subsection{Direction sets given by points on a curve} \label{points-on-curve-section} Given $n \geq 2$ and a fixed integer $D$, let $\mathcal G_D$ denote the class of continuous curves $\Gamma: I \rightarrow \R^n$ for an interval $I\subset \R$ such that 
\begin{itemize} 
\item $\Gamma$ has no self-intersections, i.e., $\Gamma(s) \ne \Gamma(t)$ for $s \ne t$, and 
\item $\Gamma$ has no more than $D$ intersections with most hyperplanes. More precisely, for Lebesgue almost every $u \in \mathbb S^n$, the hyperplane $Z(P_u) = \{y \in \mathbb R^n : u \cdot \langle y, 1 \rangle = 0 \}$ intersects $\Gamma$ at most $D$ times. 
\end{itemize} 
Let us define 
\begin{equation}  \mathfrak C(N, D; n) := \sup \left\{ \norm{H_\om}_{2 \rightarrow 2} \Bigl| \begin{aligned}  &\exists \Gamma \in \mathcal G_D \text{ such that } \Omega \subseteq \Gamma(\mathbb R) \\ &\text{ and }  \#(\Omega) \leq N \end{aligned}  \right\}. \label{CND} \end{equation} 
This type of ``one-dimensional" direction set appears in \cite{CDR}, where the authors prove a bound of the form $\mathfrak C(N, D;n) \les D \log N$.  We give a different proof of this result with a small improvement, which incidentally is also optimal. 
\begin{theorem} \label{CDR-reproof}  Let $\mathfrak C(N, D;n)$ be as in \eqref{CND}. Then there exists an absolute positive constant $C > 0$ such that for all $n, D \geq 1$, 
\begin{equation}\label{eqn:curve}
\mathfrak C(N, D;n) \leq C \sqrt{D} \log N.
\end{equation}  
\end{theorem} 
\begin{proof} 
Without loss of generality, we may assume that $N$ is a power of 2. Let us fix a curve $\Gamma \in \mathcal G_D$, $\Gamma: I \rightarrow \mathbb R^n$ and a direction set $\Omega$ of cardinality $N$, which we may write \[ \Omega = \{ \Gamma(t_k) :  1 \leq k \leq N \} \subset \Gamma, \quad \text{ with }  \quad t_1 < t_2 < \cdots < t_N. \] 
\noindent For an application of Theorem \ref{thm:ortho}, we cover $\om$ by $N/2$ disjoint connected sets $O_j = \Gamma([t_{2j-1},t_{2j}])$, 
$1 \leq j \leq N/2$. In the notation of Theorem \ref{thm:ortho}, the set $\Omega_j = O_j \cap \Omega$ consists of two elements of $\om$. Therefore the $L^2$-operator norm of $H_{\om_j}$ is at most 2. Selecting a point from each $O_j$ leads us to a set $\cO$ consisting of $N/2$ points on $\Gamma$. By the assumption $\Gamma \in \mathcal G_D$, we also have that 
\begin{align*}  E(u) &:= \# \bigl\{j : O_j \cap Z(P_u) \ne \emptyset \bigr\} \\ &\leq \#\;\; \text{intersections between $\Gamma$ and $Z(P_u)$, which is} \\ &\leq D, \quad   \text{ for almost every } u \in \mathbb S^n, \end{align*}  
which yields $\norm{E}_{L^\infty(\mathbb S^{n})} \leq D$. Substituting this into \eqref{almost-ortho} and invoking the definition \eqref{CND} of $\mathfrak C(N, D; n)$, we obtain 
\begin{equation}\label{eqn:indcurve}
\mathfrak C(N, D; n) \leq \mathfrak C(N/2, D;n) + 3\sqrt{D}.
\end{equation}
The claim \eqref{eqn:curve} now follows from \eqref{eqn:indcurve}, either by iteration or an induction on $N$. 
\end{proof} 
\noindent {\em{Remarks: }} 
\begin{enumerate}[1.]
\item The estimate \eqref{eqn:curve} is optimal, both in the exponent of $D$ and of $\log N$. We expand on this below. 
\vskip0.1in 
\item If we choose $n=1$ and $\Gamma: \mathbb R \rightarrow \mathbb R$ as the identity map, then $D = 1$. In this case, Theorem \ref{CDR-reproof} yields the well-known estimate \eqref{eqn:2d} for the maximal directional Hilbert transform in $\mathbb R^2$ associated with a general direction set $\Omega \subseteq \mathbb R$ of cardinality $N$.  The bound \eqref{eqn:2d} is sharp  \cite[Theorem 1]{Kim}, as can be seen for the uniform direction set $\Omega = U_N$ defined in \eqref{def-uniform}. This shows that the exponent of $\log N$ cannot be replaced by anything smaller than 1.  
\vskip0.1in
\item On the other hand, the power of $D$ is optimal as well. Let us choose $n= 2$, $\Omega_N = U_M^2$ with $M^2 = N$. From \cite[Theorem 2]{Kim}, we know that there exists a constant $c > 0$ such that 
\begin{equation} \label{uniform-2}
||H_{\Omega_N}||_{2 \rightarrow 2} \geq c N^{1/4}. 
\end{equation} 
Let us now define a curve $\Gamma$ that traces the points of $\Omega$ in horizontal rows, as follows, 
\begin{align*}  \Gamma &= \bigcup_{j = 1}^M \Gamma_{j} \cup \Gamma'_{j}, \text{ where } \Gamma_j =  [0,1] \times \{ {j}/{M}\},  \\  
\Gamma_j' &= \{ \epsilon_j \} \times [j/M, (j+1)/M], \quad \text{ and } \quad \epsilon_j = \begin{cases} 1 &\text{ if } j \text{ is odd } \\ 0 &\text{ if } j \text{ is even }. \end{cases}   
\end{align*} 
It is easy to see that $\Omega_N \subseteq \Gamma$. Further, any line that is not horizontal or vertical intersects $\Gamma$ in at most $M$ points, hence $\Gamma \in \mathcal G_{D}$ for $D = M = N^{1/2}$, Substituting this into \eqref{eqn:curve} yields the bound of \[ ||H_{\Omega_N}||_{2 \rightarrow 2} \leq C N^{1/4} \log N.\]  In view of \eqref{uniform-2}, this upper bound is sharp except possibly the factor of $\log N$. Hence the power of $D$ in \eqref{eqn:curve} cannot be further reduced, since any such reduction would violate \eqref{uniform-2} for this example.   
\end{enumerate} 
\vskip0.1in
\subsection{Direction sets given by special products} 
We now turn to direction sets with a larger number of independent parameters. Given $n \geq 1$, let us fix integers $N_1 \geq N_2 \geq \cdots \geq N_n \geq 1$. 
As in \eqref{CND}, we define 
\begin{equation} \label{def-Cprod} 
\begin{split}
&\mathfrak C_{\text{prod}}(N_1, \cdots, N_n; n) :=  \\ 
& \sup \left\{ ||H_{\Omega}||_{2 \rightarrow 2} \; \Bigl| \; \begin{aligned} &\Omega = \Omega_1 \times \cdots \Omega_n, \text{ where }  \Omega_j \subset \mathbb R \text{ and}  \\ & \#(\Omega_j) \leq N_j \text{ for all } 1 \leq j \leq n \end{aligned}  \right\}.  
\end{split}
\end{equation}  
The article \cite{Kim} provides sharp bounds for $\mathfrak C_{\text{prod}}(N_1, \cdots, N_1; n)$, i.e., for direction sets given by Cartesian products of sets of equal cardinalities; specifically, it is shown that for some constant $C = C_n > 0$, 
\begin{equation} \label{Kim-Cprod}
\mathfrak C_{\text{prod}}(N_1, \cdots, N_1; n) \leq \begin{cases} C\log (N_1 + 1) &\text{ if } n = 1, \\ C N_1^{\frac{n-1}{2}} &\text{ if } n \geq 2. \end{cases} 
\end{equation} 
 We will generalize this result shortly in the next section, in Theorem \ref{thm:prod}.  
As preparation for this, and as a simple illustration of the main ideas, we use Theorem \ref{thm:ortho} to reprove a result of \cite{Kim} in a special case. 
\begin{theorem} \cite[Theorem 2]{Kim} \label{Kim-theorem} 
Let $n=2$. Then there exists an absolute constant $C > 0$ such that the quantity $\mathfrak C_{\text{prod}}$ defined in \eqref{def-Cprod} obeys the estimate:
\begin{equation} \mathfrak C_{\text{prod}}(N_1, N_1; 2) \leq C \sqrt{N_1}. \label{Kim-dim2} \end{equation}   
\end{theorem} 
\begin{proof} 
Without loss of generality, we may assume $N_1$ is a power of 2, i.e., of the form $N_1 = 2^r$, $r \geq 0$. We will prove \eqref{Kim-dim2} by induction on $r$ 
 with $C= 5/(1-2^{-1/2})$. For the base case $r = 0$ or $N_1= 1$, the statement is valid since $\mathfrak C_{\text{prod}}(1, 1; 2) = 1$ and $C \geq 1$. For the inductive step, we assume that \eqref{Kim-dim2} holds for all integers $N_1 = 2^r$ with $r < R$. We aim to prove \eqref{Kim-dim2} for $N_1 = 2^{R}$. Accordingly, let us fix a direction set $\om = \Omega_1 \times  \Omega_2$ of cardinality $N= 2^{2R}$, where both $\Omega_1, \Omega_2 \subset \R$ have cardinality $2^{R}$. 
 \vskip0.1in
\noindent Let $\{I_k = (\alpha_{k-1}, \alpha_{k}) : 1 \leq k \leq N_1/2= 2^{R-1} \}$ be a finite cover of $\Omega_1$ consisting of a collection of disjoint intervals in $\mathbb R$, with $\alpha_0 = -\infty$, $\alpha_{N_1/2} = \infty$, and each $I_k$ containing exactly two consecutive elements of $\Omega_1$.  Let $\{ J_{\ell} = (\beta_{\ell-1}, \beta_{\ell}): 1 \leq \ell \leq 2^{R-1} \}$ be a similar cover for $\Omega_2$. Based on these, we choose the axes-parallel rectangles $\{ O_{k \ell} = I_k \times J_{\ell} : 1 \leq k, \ell \leq 2^{R-1} \}$, which will serve as the connected sets $O_j$ required by Theorem \ref{thm:ortho}. Clearly, the sets $\{O_{k \ell} \}$ form a finite cover of $\Omega$, and each set $\Omega_{k \ell} = \Omega \cap O_{k \ell}$ contains exactly 4 points, so that 
\begin{equation} \label{4-element-omkl} 
||H_{\Omega_{k \ell}} ||_{2 \rightarrow 2} \leq 4.
\end{equation} 
Let us also record here that 
\begin{equation} 
||E||_{L^{\infty}(\mathbb S^2)} := \sup_{u \in \mathbb S^2} \#\bigl\{(k, \ell) : Z(P_u) \cap O_{k \ell} \ne \emptyset \bigr\} \leq N_1 = 2^R, \label{E-bound}
 \end{equation} 
 in the notation of Theorem \ref{thm:ortho}. In fact, $Z(P_u)$ is an affine line, and any line $L$ in $\mathbb R^2$ intersects at most $N_1-1$ of the sets $O_{k \ell}$, i.e.
\begin{equation} 
\# \bigl\{(k, \ell) : L \cap O_{k \ell} \ne \emptyset \bigr\} \leq N_1-1,  \label{Example-E} 
\end{equation} 
which implies \eqref{E-bound}. An elementary proof of the geometric statement \eqref{Example-E} can be found in the Lemma \ref{elementary-geometric-lemma} below.  
\vskip0.1in
\noindent Assuming \eqref{Example-E} for now, let us choose points $a_k \in I_k$ and $b_{\ell} \in J_{\ell}$, and set \[ \Omega_1' = \{ a_k : 1 \leq k \leq 2^{R-1}\}, \quad \Omega_2' = \{ b_{\ell} : 1 \leq \ell \leq 2^{R-1} \}, \quad \mathcal O = \Omega_1' \times \Omega_2'. \]
We observe that $\mathcal O$ is a product of two sets, each of size $N_1/2 = 2^{R-1}$; hence invoking Theorem \ref{thm:ortho} with \eqref{4-element-omkl} and \eqref{E-bound} yields 
\[ ||H_{\Omega}||_{2 \rightarrow 2} \leq ||H_{\mathcal O}||_{2 \rightarrow 2} + 2^{\frac{R}{2}} (4+1) \leq \mathfrak C_{\text{prod}}(2^{R-1}, 2^{R-1}; 2) + 5 2^{R/2}. \] 
Taking supremum of the left hand side above over all product sets $\Omega = \Omega_1 \times \Omega_2$ with $\#(\Omega_1) = \#(\Omega_2) = N_1$ and applying the induction hypothesis, we obtain   
\[ \mathfrak C_{\text{prod}}(2^R, 2^R; 2) \leq \mathfrak C_{\text{prod}}(2^{R-1}, 2^{R-1}; 2) + 5 2^{R/2} \leq (C 2^{-1/2}+5)\sqrt{2^R} = C 2^{R/2}, \]
by our choice of $C$. This closes the induction and completes the proof. 
\end{proof}

\begin{lem} \label{elementary-geometric-lemma} 
In the setup described in the proof of Theorem \ref{Kim-theorem}, the estimate \eqref{Example-E} holds for every line $L$ in $\mathbb R^2$.  
\end{lem} 
\begin{proof} 
Let us consider the polynomial 
\begin{equation}  P(x_1, x_2) = \prod_{k=1}^{N_1/2-1} (x_1 - \alpha_k) \prod_{\ell=1}^{N_1/2-1} (x_2 - \beta_{\ell}). \label{What-is-P} \end{equation}  
It is clear that the sets $O_{k \ell}$ defined in the proof of Theorem \ref{Kim-theorem} are the ``cells", or connected components, of $\mathbb R^2 \setminus Z(P)$. The notation $Z(P)$ represents the zero set of $P$, which in this case consists of $N_1/2-1$ vertical lines $\{ \alpha_k \} \times \mathbb R$ and $N_1/2-1$ horizontal lines $\mathbb R \times \{\beta_{\ell} \}$, with $1 \leq k, \ell \leq N_1/2-1$. If $L$ is a horizontal or a vertical line, it is clear that it intersects at most $N_1/2$ of the cells $O_{k \ell}$. Since $N_1/2 \leq N_1-1$, the inequality \eqref{Example-E} follows immediately in this case. If $L$ is not such a line, then $L$ is given by an equation of the form $x_2 = a x_1 + b$ for some nonzero, finite slope $a$. Substituting this into the expression \eqref{What-is-P} for $P$ leads to a univariate polynomial in $x_1$ of degree $2(N_1/2-1) = N_1-2$. This means that $L$ intersects at most $N_1-2$ points in $Z(P)$; in other words, $L$ can intersect at most $N_1-1$ cells $O_{k \ell}$. This provides the required estimate \eqref{Example-E}, completing the proof.     
\end{proof} 
\section{Applications and new results: Part 2} \label{Example-section-2} 
\noindent The results here lie in four largely unrelated directions, except for the common theme that Theorem \ref{thm:ortho} appears in all their proofs. We present them in separate subsections. Proofs are often relegated to later sections. 
\subsection{Maximal directional Hilbert transforms in $\mathbb R^2$ for direction sets of mixed type}  The notion of finite order lacunarity in $\mathbb R$, in connection with directional operators, first appears in the work of Sj\"ogren and Sj\"olin \cite{Sjogren-Sjolin}. We refer the reader to this article for the relevant definitions. Direction sets $\Theta \subseteq \mathbb R$ that are lacunary of finite order  play a key role in the study of the directional maximal average $M_\Theta$ defined in \eqref{max-dir}. For example, if $\Theta \subseteq \mathbb R$ is an infinite direction set, the following dichotomy is known \cite{{Sjogren-Sjolin}, {Alf}, {Bateman}} for $M_{\Theta}$ in $\mathbb R^2$: 
\vskip0.1in
\begin{enumerate}[1.] 
\item Suppose there exists $0 < \lambda < 1$ and $R \geq 1$ such that $\Theta$ can be covered by a finite union of sets, each of which is lacunary of order at most $R$ with lacunarity constant at most $\lambda$. Then $M_{\Theta}$ is bounded on $L^p(\mathbb R^2)$ for all $1 < p \leq \infty$.  
\vskip0.1in 
\item Suppose that $\Theta$ does not admit a finite cover of the type mentioned above. Then $M_{\Theta}$ is unbounded on $L^p(\mathbb R^2)$ for all $1 \leq p < \infty$. 
\end{enumerate} 
\vskip0.1in
The situation for the maximal directional average in $\mathbb R^2$ leads one to consider the possibility of a similar dichotomy for the maximal directional Hilbert transform, suitably interpreted. Of course $H_{\Theta}$ is unbounded on $L^p$ for all $p \in (1, \infty)$ since $\Theta$ is infinite, but it is of interest to quantify the growth rates of the operator norms associated with various finitary exhaustions of $\Theta$. In particular, the uniform lower bound \eqref{maxDHT-lowerbound} and the uniform upper bound \eqref{eqn:2d} prompt the following natural questions:
\vskip0.1in
\noindent {\em{Question 1:}} Is there a classification of the blow-up rates for the maximal directional Hilbert transform, depending on the intrinsic geometric structure of the direction set? More precisely, suppose that $S \subseteq [1/2, 1]$ denotes the set of ``possible blow-up exponents'' in $\mathbb R^2$; explicitly stated, $S$ consists of all exponents $\alpha$ such that there exists an infinite direction set $\Theta$, and a choice of a sequence \label{Questions}
\begin{equation} \label{Theta-sequence} 
\Theta_1 \subsetneq \Theta_2 \subsetneq \cdots \Theta_N \subsetneq \cdots \Theta \subseteq \mathbb R
\end{equation} 
of increasing finite subsets of $\Theta$, such that 
\[ 0 <  \liminf_{N \rightarrow \infty}\frac{||H_{\Theta_N}||_{2 \rightarrow 2}}{(\log \#(\Theta_N))^{\alpha}} \leq \limsup_{N \rightarrow \infty}\frac{||H_{\Theta_N}||_{2 \rightarrow 2}}{(\log \#(\Theta_N))^{\alpha}} < \infty. \] 
Can one give a complete description of $S$? 
\vskip0.1in
\noindent In Theorem \ref{thm:growth-alpha} below, we show that $S = [1/2,1]$, i.e., every number in $[1/2, 1]$ is realizable as a blow-up exponent of $||H_{\Theta_N}||_{2 \rightarrow 2}$ for an appropriate choice of $\Theta \subseteq \mathbb R$.  
\vskip0.1in 
\noindent {\em{Question 2:}} Does finite order lacunarity play a distinguished role for the maximal directional Hilbert transform as well? For instance, in the notation of \eqref{Theta-sequence}, does the blow-up rate \begin{equation} ||H_{\Theta_N}||_{2 \rightarrow 2} \sim \sqrt{\log \#(\Theta_N)} \label{blowup-half} \end{equation}  imply that $\Theta$ has to be lacunary of finite order?     
\vskip0.1in 
\noindent In Corollary \ref{cor-dim2}, we answer this question in the negative, by constructing an infinite set $\Theta$ that is not lacunary of any finite order, which permits an increasing sequence of finite subsets $\Theta_N$ obeying \eqref{blowup-half}. 
\vskip0.1in 
\begin{theorem} \label{thm:growth-alpha} 
For any exponent $\alpha \in [\frac12, 1]$, there exists an infinite direction set $\Theta = \Theta(\alpha)$ and subsets $\Theta_N = \Theta_N(\alpha)$, with 
\begin{equation}  \Theta_1 \subsetneq \Theta_2  \subsetneq \cdots \subsetneq \Theta_N \subsetneq \cdots \subsetneq \Theta, \quad \#(\Theta_N) \nearrow \infty, \label{conditions-ThetaN} \end{equation} 
such that 
\begin{equation} \label{ThetaN-op-norm}   C^{-1} (\log \#(\Theta_N))^{\alpha} \leq ||H_{\Theta_N}||_{L^2(\mathbb R^2) \rightarrow L^2(\mathbb R^2)} \leq C (\log \#(\Theta_N))^{\alpha}. \end{equation}   
Here $C > 0$ is an absolute constant, independent of $\alpha$. 
\end{theorem} 
\noindent {\em{Remark: }} The above result is planar. In Corollary \ref{prescribed-growth-cor} of the next subsection, we prove an analogous result in dimensions three and higher.  For every $n \geq 2$ and any choice of $\alpha \in (0, \frac{n-1}{2n})$ and $\beta \geq 0$, we find an increasing sequence of sets $\Theta_N \subsetneq \Theta \subseteq \mathbb R^n$, such that $||H_{\Theta_N}||_{2 \rightarrow 2}$ goes to infinity at the rate of $ (\# \Theta_N)^{\alpha} (\log \# \Theta_N)^{\beta}$.
\vskip0.1in 
\begin{proof} 
For a given exponent $\alpha \in [\frac12, 1]$, we choose an increasing sequence of positive integers \[ M_1 \ll R_1 \ll M_2 \ll R_2 \ll \cdots M_N \ll R_N \ll \cdots \]  such that $M_N$ divides $M_{N+1}$ for every $N$, and 
\begin{equation} \label{logMN-logN}  \frac{1}{2} (\log R_N)^{\alpha} \leq \log M_N \leq (\log R_N)^{\alpha} \quad \text{ for all } N \geq 1. \end{equation}  
Let us set 
\begin{equation} \label{def-Theta-ThetaN} \Theta_N := \bigcup_{j=1}^{R_N} 2^{-j} + 2^{-j} U_{M_N}, \quad \text{ and } \quad \Theta = \bigcup_{N=1}^{\infty} \Theta_N. \end{equation}  
where $U_M$ is the uniform direction set given by \eqref{def-uniform}. Thus each $\Theta_N$ is an $R_N$-fold union of affine copies of the uniform direction set $U_{M_N}$; each copy is arranged within the successive elements of a finite lacunary sequence $\{2^{-j}, 1 \leq j \leq R_N\}$.  The fact that $M_{N}$ is an integer multiple of $M_{N-1}$ ensures that $U_{M_{N-1}} \subsetneq U_{M_N}$. Hence the sets $\Theta_N$ obey the inclusion relation in \eqref{conditions-ThetaN}, with $\#(\Theta_N) = R_NM_N \nearrow \infty$. In view of \eqref{logMN-logN} and the restriction $\alpha \leq 1$, we observe that  
\begin{equation} \label{logOmega_N} 
\log R_N \leq \log \#(\Theta_N) = \log R_N + \log M_N \leq 2 \log R_N. 
\end{equation} 
\vskip0.1in
\noindent To estimate $||H_{\Theta_N}||_{2 \rightarrow 2}$, we apply Theorem \ref{thm:ortho} with $n=1$, $\Omega = \Theta_N$, $O_j = [2^{-j}, 2^{-j+1})$, $1 \leq j \leq R_N$, and $\mathcal O = \{2^{-j} : 1\leq j \leq R_N \}$.  As a result $\Omega_j = 2^{-j} + 2^{-j} U_{M_N}$, which is an affine copy of $U_{M_N}$. As discussed in item \ref{remark-n1} of the remarks following Theorem \ref{thm:ortho}, an application of \eqref{almost-ortho-n1} yields 
\begin{align*}
||H_{\Theta_N}||_{2 \rightarrow 2} &\leq ||H_{\mathcal O}||_{2 \rightarrow 2} + \max_{j} ||H_{\Omega_j}||_{2 \rightarrow 2} + 1\\
&\leq C (\sqrt{\log R_N} + \log M_N) + 1 \\ &\leq C \bigl[ (\log R_N)^{\frac{1}{2}}  + (\log R_N)^{\alpha} \bigr] \\& \leq C (\log R_N)^{\alpha} \leq C (\log \#(\Theta_N))^{\alpha}. 
\end{align*} 
In the second inequality, we have used two known results: 
\[ ||H_{\mathcal O}||_{2 \rightarrow 2} \leq C \sqrt{\log \#(\mathcal O)} \;\; \text{ and } \;\; ||H_{\Omega_j}||_{2 \rightarrow 2} \leq  C \log \#( \Omega_j) \text{ for } 1 \leq j \leq N.  \] The first estimate follows from the work of Demeter and Di Plinio \cite{DD}, and has been mentioned in \eqref{eqn:lac}. The second estimate is a consequence of the general estimate \eqref{eqn:2d}. Invoking \eqref{logMN-logN} and \eqref{logOmega_N} leads to the final expression.  This establishes the right hand inequality in \eqref{ThetaN-op-norm}. 
\vskip0.1in  
\noindent To establish the left hand inequality in \eqref{ThetaN-op-norm}, we observe that $\Theta_N \supseteq 1/2 + (1/2) U_{M_N}$. This leads to  
\[ ||H_{\Theta_N}||_{2 \rightarrow 2} \geq ||H_{2^{-1} + 2^{-1} U_{M_N}}||_{2 \rightarrow 2} =  ||H_{U_{M_N}}||_{2 \rightarrow 2} \]
by the invariance of the operator norm of maximal directional Hilbert transform under affine transformations of the direction set: Lemma \ref{lem:modify}. By \cite[Theorem 1]{Kim}, the last quantity is bounded below by a constant multiple of $\log M_N$. In view of \eqref{logMN-logN} and \eqref{logOmega_N}, we have that $\log M_N \geq (\log R_N)^{\alpha}/2 \geq (\log \#(\Theta_N))^{\alpha}/4$. This completes the proof of the theorem. 
\end{proof} 
\begin{cor} \label{cor-dim2} 
There exists an infinite set $\Theta \subseteq \mathbb R$ with the following properties:
\begin{enumerate}[(a)]
\item \label{alpha-half-parta} There does not exist any $\lambda < 1$ or $1 \leq R < \infty$ such that $\Theta$ can be covered by finitely many lacunary sets of order at most $R$ and lacunarity constant $\lambda <1$. 
\item \label{alpha-half-partb} There exists an exhaustion of $\Theta$ by an increasing sequence of finite sets $\Theta_N$ such that 
\[ C^{-1} \sqrt{\log \#(\Theta_N)} \leq ||H_{\Theta_N}||_{L^2(\mathbb R^2) \rightarrow L^2(\mathbb R^2)} \leq C \sqrt{\log \#(\Theta_N)}. \] 
\end{enumerate} 
\end{cor} 
\begin{proof}
For $\alpha = 1/2$, let us choose $\Theta$ and $\Theta_N$ as in \eqref{def-Theta-ThetaN}. The conclusion of part \eqref{alpha-half-partb} of the corollary then follows from  \eqref{ThetaN-op-norm} in Theorem \ref{thm:growth-alpha}. The lack of finite order lacunarity of $\Theta$ is well-known and can either be verified directly from the definition in \cite{Sjogren-Sjolin} or by computing the splitting number of the binary tree depicting $\Theta$, as in \cite{Bateman}, and verifying that this quantity is infinite. For example, if each $M_N$ is a power of 2, then the splitting number of $\Theta_N$, and hence $\Theta$, is at least $\log_2 M_N$. The proof of this latter fact, which may be found in \cite{Bateman}, involves ideas largely unrelated with the main theme of this paper, and we choose to omit it here.   
\end{proof} 
 \subsection{Direction sets of general product type} As a consequence of Theorem \ref{thm:ortho}, we are able to extend Theorem \ref{Kim-theorem} to include direction sets given by Cartesian products of finite sets, where the finite sets are allowed to have different cardinalities. We recall that $\mathfrak C_{\text{prod}}(N_1, \cdots, N_n; n)$ is defined in \eqref{def-Cprod}.  
\begin{theorem} \label{thm:prod}
For every $n\geq 2$, there exists an absolute constant $C = C_n > 0$ such that for any choice of integers $N_1 \geq N_2 \geq \cdots \geq N_n \geq 1$, the following estimate holds:
\begin{equation}   \mathfrak C_{\text{prod}}(N_1, \cdots,  N_n; n) \leq C \Bigl[\prod_{k=2}^{n} N_k \Bigr]^{1/2} \log\Bigl(\frac{N_1}{N_2}+1\Bigr). \label{eqn:prod:norm} \end{equation} 
The bound is sharp; the reverse inequality holds with a different implicit constant $C$ for all direction sets of the form $\om= \prod_{k=1}^n U_{N_k}$, where $U_M$ is as in \eqref{def-uniform}.
\end{theorem}
\noindent {\em{Remarks: }} 
\begin{enumerate}[1.] 
\item We illustrate the estimate in the case $n=2$. When  $N_1=N$ and $N_2 = 1$, Theorem \ref{thm:prod} shows that $\mathfrak C_{\text{prod}}(N, 1; 2) \leq C \log N$, recovering the $\log N$ bound in \eqref{eqn:2d} for $n=1$, in view of Lemma \ref{lem:slice}. When $N_1 = N_2$, Theorem \ref{thm:prod} recovers the $\sqrt{N_1}$ bound from \eqref{Kim-dim2}. 
\vskip0.1in

\item Theorem \ref{thm:prod} has been proved in section \ref{sec:prod}. 

\end{enumerate} 
\subsection{Maximal directional Hilbert transforms of prescribed growth} 
Theorem \ref{thm:prod} provides examples of sets of the form $\Omega = U_{N_1} \times U_{N_2} \times \cdots \times U_{N_n}$ with $N_1 N_2 \cdots N_n = N$ such that the $L^2$-operator norm of $H_{\Omega}$ exhibits growth rates of order $N^{\alpha}$  for every $0 < \alpha < (n-1)/(2n)$. This generalizes Theorem \ref{thm:growth-alpha} to the setting where $n \geq 2$. We state this observation as a corollary.
\begin{cor} \label{prescribed-growth-cor}
For $n \geq 2$, let us fix parameters $\alpha, \beta$ with $\alpha \in (0, \frac{n-1}{2n})$ and $\beta \in [0, \infty)$. Then there exists a constant $C= C(\alpha) > 0$ such that for every sufficiently large integer $N\geq N_0(\alpha, \beta)$, there is a direction set $\om \subset \R^n$, given by an $n$-fold Cartesian product of uniform sets $U_M$ of the form \eqref{def-uniform} that obeys the following conclusions:
\begin{align} 
N/2^n &\leq \#(\Omega) \leq 2^nN, \text{ and } \label{size-Omega} \\  
C^{-1} N^\alpha (\log N)^\beta &\leq  \norm{H_\om}\opn{n+1} \leq C N^\alpha (\log N)^\beta. \label{norm-MDHT}  
\end{align}  
\end{cor}
\begin{proof}
Since $0 < \alpha < (n-1)/(2n)$, we can choose $N$ sufficiently large depending on $\alpha$ and $\beta$ so that 
\begin{equation} 
N^{\frac{n-1}{2n} - \alpha} > 4^{\frac{n-1}{2n}} (\log N)^{(\beta -1)}. 
\label{N-large} 
\end{equation} 
 Let us choose $\om = U_{N_1}\times U_{N_2} \times \cdots U_{N_n}$, where  $N_2 = N_3 = \cdots = N_n$, and  
\begin{equation} \label{size-N1N2}
\begin{split}
  \frac{N_j}{2} \leq N^{\frac{2 \alpha}{n-1}} (\log N)^{\frac{2(\beta -1)}{n-1}} &\leq 2N_j \text{ for } j \geq 2 \; \text{ and } \\
   \frac{N_1}{2} \leq N^{1- 2\alpha} (\log N)^{2(1-\beta)} &\leq 2N_1.  
  \end{split}
\end{equation} 
The condition \eqref{N-large} implies that $N_1 > N_2 = N_3 =  \cdots = N_n \geq 1$, so that the requirements of Theorem \ref{thm:prod} are met. It also implies that $\#(\Omega) = N_1 N_2 \cdots N_n$ satisfies \eqref{size-Omega}.  In addition, the assumptions in \eqref{size-N1N2} ensure that  
\begin{equation}  \frac{N_1}{4N_2} \leq N^{1- \frac{2 \alpha n}{n-1}} (\log N)^{\frac{2n}{n-1} (1-\beta) } \leq \frac{4N_1}{N_2}. \label{log-size}  \end{equation}  
From the inequalities in \eqref{size-N1N2} and \eqref{log-size}, and invoking the size restrictions on $N$ provided by \eqref{N-large}, we can thus find a constant $C = C(\alpha, n) > 0$ such that 
\begin{equation}  C^{-1} N^{\alpha} (\log N)^{\beta} \leq \bigl[ \prod_{k=2}^{n}N_k \bigr]^{\frac{1}{2}} \log \Bigl( \frac{N_1}{N_2} + 1\Bigr) \leq C N^{\alpha} (\log N)^{\beta}  \label{prod-estimate}\end{equation} 
According to Theorem \ref{thm:prod}, $||H_{\Omega}||_{2 \rightarrow 2}$ is of size comparable to the middle term, and hence to all three terms in \eqref{prod-estimate}. This leads to the desired conclusion \eqref{norm-MDHT}.  
\end{proof}


\subsection{Improved estimates for the maximal directional Hilbert transform in $\mathbb R^3$} 
In the special case when $n=2$, i.e., in dimension 3, some of the results in this paper can be sharpened. The first such example is an improvement of Theorem \ref{thm:highd}. 
\begin{theorem} \label{thm:3dgen} Suppose that $h:[2,\infty) \to [2,\infty)$ is an increasing function, $h(N) \nearrow \infty$, such that 
\begin{equation} \label{conditions-on-h}  \lim_{N\to \infty} \frac{h(N)}{\log N} = 0, \;\; \text{ and } \;\; \lim_{N\to \infty} \frac{h \left((\log N)^4 \right) }{h(N)} = 0. \end{equation} 
Then there is a positive constant $C$ depending only on $h$ such that for every $N\geq 2$, and any direction set $\Omega \subseteq \mathbb R^2$ with $\#(\Omega) = N$, we have 
\begin{equation}  \norm{ H_\om}_ {L^2(\mathbb{R}^3) \to L^2(\mathbb{R}^3)} \leq C N^{1/4} h(N). \label{eqn:3dgen} \end{equation} 
\end{theorem}
\noindent {\em{Remark: }} We present the proof of Theorem \ref{thm:3dgen} in section \ref{sec:3d}. 
\vskip0.1in  
\noindent The improvement obtained in Theorem \ref{thm:3dgen} relies, in turn, on an estimate for $H_{\Omega}$ where $\Omega$ is an algebraic set in $\mathbb R^2$.  More precisely, suppose that  $\mathcal P_d(2)$ denotes the collection of all real polynomials $P \in \mathbb R[x_1, x_2]$ of degree at most $d$ such that $P \not\equiv 0$, and let $Z_{\mathbb R}(P) := \{ x = (x_1, x_2) \in \R^2: P(x) = 0 \}$. As in \eqref{CND}, let us define 
\begin{equation} \label{C-star-N-d}
\mathfrak C_2^{\ast}(N; d) := \sup \left\{  \norm{H_\om}\opn{3} \Bigl| \; \begin{aligned}  &\exists P \in \mathcal P_d(2), \text{ such that } \\  &\om \subset Z_{\mathbb R}(P), \; \# \om \leq N  \end{aligned} \right\}.    
\end{equation} 
\begin{theorem} \label{thm:3d} There is a positive absolute constant $A$ such that for any $d\geq 1$ and $N\geq 3,$ 
\begin{equation} \label{Zp-3d} \mathfrak C_2^{\ast}(N; d)   \leq A \sqrt{d} \log N. \end{equation} 
\end{theorem}
\noindent {\em{Remark: }} 
\begin{enumerate}[1.]
\item Setting $n=2$ and $m=1$ in Theorem \ref{thm:higherdim}  gives that $\mathfrak C_2^{\ast}(N; d) \leq C(\epsilon, d) N^{\epsilon}$.  In this sense, Theorem \ref{thm:3d} may be viewed as an improvement of Theorem \ref{thm:higherdim} in the case $n=2$; it quantifies the dependence on $d$ and replaces $N^{\epsilon}$ by $\log N$.    

\vskip0.1in
\item It is well-known \cite[Lemma 2.4]{Guth}  that for any set of $N$ points in $\mathbb \R^2$, there exists a nontrivial polynomial $P$ of degree $\leq d$ that vanishes on this set, provided $N \leq \binom{d+1}{2}$. In particular, given any $\Omega \subseteq \mathbb R^2$ of cardinality $N$, we can always choose $P \in \mathcal P_d(2)$ with $d = 2 \sqrt{N}$ such that $\Omega \subseteq Z(P)$. Thus, while the estimate \eqref{Zp-3d} is ostensibly for all $N \geq 2$ and $d \geq 1$, in practice one has the relation 
\begin{equation} \label{C2-dlarge-N} 
\mathfrak C_{2}^{\ast}(N; d) = \mathfrak C_{2}^{\ast}(N; 2 \sqrt{N}), \quad \text{ for all }  d \geq 2\sqrt{N}, 
\end{equation} 
which offers a better bound; namely $\mathfrak C_2^{\ast}(N, d) \leq A \sqrt{2} N^{1/4} \log N$. One may therefore rephrase \eqref{Zp-3d} as follows: \[  \mathfrak C_2^{\ast}(N; d)   \leq A \bigl[\min(d, \sqrt{N}) \bigr]^{\frac{1}{2}} \log N. \]  
\vskip0.1in
\item In view of the previous remark, \eqref{Zp-3d} already leads to Theorem \ref{thm:highd} for $n=2$ with an improvement. Specifically, for any $\Omega \subseteq \mathbb R^2$ of cardinality $N$, Theorem \ref{thm:3d} gives that 
\begin{equation}\label{eqn:bound0}
||H_{\Omega}||_{2 \rightarrow 2} \leq C N^{1/4} \log N.
\end{equation}
While the conclusion of Theorem \ref{thm:3dgen} is stronger, we will see that it uses Theorem \ref{thm:3d} as a crucial ingredient.  
\vskip0.1in

\item Theorem \ref{thm:3d} is sharp in the sense that it fails to hold if either the exponent $1/2$ of $d$ or the exponent 1 of $\log N$ is replaced by a smaller quantity. The sharpness of the exponent $1/2$ of $d$ follows from the sharpness of the bound \eqref{eqn:bound0} (up to the factor of $\log N$) in view of \eqref{uniform-2}. Moreover, we observe that the direction set $\Omega = U_N \times \{0\}$ meets the requirement of Theorem \ref{CDR-reproof} with $n=2$, $P(x_1, x_2) = x_2$ and $d = 1$. Invoking Lemma \ref{lem:slice} and \cite[Theorem 1]{Kim}, we find that $||H_{\Omega}||_{2 \rightarrow 2} \geq c \log N$ for some constant $c > 0$. 

\vskip0.1in 
\item We ask the reader to compare the statements of Theorem \ref{thm:3d} and Theorem \ref{CDR-reproof} in dimension 3. They looks 
similar, but each encodes information not completely captured by the other. On one hand, any curve without self-intersections that is implicitly defined by a polynomial $P \in \mathcal P_d(2)$ can be intersected by a hyperplane in  at most $d$ points, and hence is in $\mathcal G_d$. However, a general curve in $\mathcal G_d$ need not be given by the zero set of a polynomial. On the other hand, the zero set of a polynomial $P \in \mathcal P_d(2)$ is in general a union of points and curves, and need not always obey the requirements of Theorem \ref{CDR-reproof}.   
\vskip0.1in 

\item The proof of Theorem \ref{thm:3d} appears in section \ref{proof:thm3d:section}. 
\end{enumerate}

\section{The almost-orthogonality principle: Proof of Theorem \ref{thm:ortho}} \label{sec:ortho}
\noindent Let us write $\vec{v} :=\inn{v,1}$ for $v\in \R^n$, and denote by $\vec{v}^\perp$ the hyperplane orthogonal to $\vec{v}$, i.e., $\vec{v}^{\perp} := \{ x \in \R^{n+1} : \vec{v} \cdot x = 0\}$ . Given $v_1,v_2\in \R^{n}$, let \[ \mathfrak L(v_1,v_2) :=\{ (1-t)v_1 + tv_2 \in \R^n : t\in [0,1] \} \]  be the finite line segment between $v_1$ and $v_2$. Throughout this section, we let $\Omega$, $O_j$, $\Omega_j$, $v_j$ and $\mathcal{O}$ be as in the statement of Theorem \ref{thm:ortho}. 
\subsection{Ingredients of the proof} 
\begin{lem} \label{lem:supp} The multiplier for the operator $(H_{v_1}  - H_{v_2})$ is supported in the set $\bigcup_{v\in \mathfrak L(v_1,v_2)} \vec{v}^\perp $.
\end{lem}
\begin{proof}
The multiplier for the operator $(H_{v_1}  - H_{v_2})$ equals 
\[ -i \left(\sgn(\xi \cdot \vec{v}_1 ) - \sgn(\xi \cdot \vec{v}_2 )\right) .\] 
Suppose that $\sgn(\xi \cdot \vec{v}_1 ) - \sgn(\xi \cdot \vec{v}_2 ) \neq 0$. Without loss of generality, we may assume that $\xi \cdot \vec{v}_1  \geq 0$ and $\xi\cdot \vec{v}_2 <0$. Therefore there exists $t\in [0,1]$ such that 
\[ (1-t) \xi \cdot \vec{v}_1 + t \xi\cdot \vec{v}_2 = 0, \quad \text{ i.e., } \quad \xi \cdot \bigl[ (1-t) \vec{v}_1 + t \vec{v}_2 \bigr] = 0. \]
It follows that $\xi \in \vec{v}^\perp$, where $v = (1-t)v_1 + t v_2  \in \mathfrak L(v_1, v_2)$.
\end{proof}


\begin{lem} \label{cor:opnorm} Set 
\begin{equation} \label{def-Wj}
W_j := \bigcup_{v \in \mathfrak L_j } \vec{v}^\perp ,
\end{equation} 
where $\mathfrak L_j $ is the union of line segments
\begin{equation} \label{def-Lj} 
\mathfrak L_j := \bigcup_{v \in \Omega_j} \mathfrak L(v,v_j).
\end{equation} 
Let $R_{W_j} f$ be the Fourier restriction operator 
\begin{equation} \label{def-RWj}
\widehat{R_{W_j} f} := \ind_{W_j} \wh{f},
\end{equation} 
where $\mathbf 1_W$ stands for the indicator function of $W$.  Then the following pointwise bound holds: 
\begin{equation}  H_\om f(x)  \leq H_\cO f(x) + \max_j H_{\om_j} \circ R_{W_j}f (x) + \max_j |H_{v_j} \circ R_{W_j}f (x)|. \label{Hom-pointwise-bound} \end{equation} 
\end{lem}

\begin{proof}
Suppose that $v\in \om_j$. By Lemma \ref{lem:supp}, we observe that
\[ (H_v - H_{v_j}) f = (H_v - H_{v_j}) R_{W_j} f. \]
Therefore, for $v\in \om_j$, 
\begin{align*}
|H_v f(x)| &\leq |H_{v_j} f(x)|  + |(H_v - H_{v_j}) R_{W_j} f(x)| \\
 &\leq H_{\cO} f(x)  + |H_v R_{W_j} f(x)| + |H_{v_j} R_{W_j} f(x)|.
 \end{align*}
Fixing the index $j$ and taking the supremum of both sides of the inequality above over $v\in \om_j$, we obtain
\[ H_{\om_j} f(x) \leq H_\cO f(x) + H_{\om_j} R_{W_j} f(x) + |H_{v_j} R_{W_j} f(x)|.\]
Taking the maximum over $j$ then finishes the proof of the pointwise bound \eqref{Hom-pointwise-bound}. 
\end{proof} 

\begin{lem} \label{lem:inters}
Let $P_u$ be defined as in \eqref{def-Pu}, with $u \in \mathbb S^n$,  and let $j$ be an index. If $Z(P_u)$ intersects the union of line segments $\mathfrak L_j$ given by \eqref{def-Lj}, then $Z(P_u)$ intersects $O_j$.
\end{lem}
\begin{proof}
We prove the contrapositive. Suppose that $Z(P_u)$ does not intersect $O_j$. Since $O_j$ is connected, it must therefore lie inside exactly one of the half-spaces $Z_{\pm}(P_u)$, where  \begin{align*} Z_+(P_u) &:= \{ y\in \R^{n} : P_u(y) >0 \}, \\ Z_-(P_u) &:= \{ y\in \R^{n} : P_u(y) <0 \}. 
\end{align*} 
Without loss of generality, suppose that $O_j \subset Z_+(P_u)$. We will show in the paragraph below that $\mathfrak L_j \subset Z_+(P_u)$, i.e., $P_u(v)>0$ for every $v\in \mathfrak L_j$. This in turn will show that $Z(P_u)$ does not intersect $\mathfrak L_j$, establishing the desired conclusion. 
\vskip0.1in
\noindent To this end, choose any $v\in \mathfrak L_j$. Then there exists $v_j'\in \om_j$ such that $v\in \mathfrak{L}(v_j',v_j)$. Thus, there exists $t\in [0,1]$ such that $v=(1-t)v_j' + tv_j$. Since $v_j',v_j \in O_j \subset Z_+(P_u)$, 
it follows that $P_u(v_j')$ and $P_u(v_j)$ are positive. Equivalently stated, the function
\begin{equation*}
s \in \mathbb R \longmapsto Q(s) := P_u((1-s) v_j' + s v_j) = \langle (1-s) v_j' + sv_j, 1 \rangle \cdot u  
\end{equation*} 
takes positive values at $s = 0$ and $s=1$. Since $Q$ is an affine linear function of $s$, it follows that the positivity is preserved for all intermediate values of $s$, in particular for $s = t$. Thus $Q(t) = P_u(v)>0$, as desired. 
\end{proof}

\subsection{Completion of the proof of Theorem \ref{thm:ortho}} 
\begin{proof} 
Let $W_j$ and $R_{W_j}$ be as in \eqref{def-Wj} and \eqref{def-RWj} respectively, so that the conclusion of Lemma \ref{cor:opnorm} holds. Applying the triangle inequality for the $L^2(\mathbb R^{n+1})$ norm on the pointwise estimate \eqref{Hom-pointwise-bound}, we arrive at:
\begin{equation} \label{after-triangle}
||H_{\Omega}f||_{2} \leq ||H_\cO f ||_2 + ||\max_j H_{\om_j} \circ R_{W_j}f||_2 + ||\max_j |H_{v_j} \circ R_{W_j}f| ||_2.
\end{equation} 
The first summand on the right hand side above corresponds exactly with the same in \eqref{almost-ortho}, so we focus on estimating the second and third summands in \eqref{after-triangle}.  
\vskip0.1in
\noindent For both the second and the third term, we bound the maximum in $j$ by the $l^2$ sum;
\begin{align*} 
\norm{\max_j H_{\om_j} R_{W_j} f}_{2}^2 &\leq \sum_j \norm{H_{\om_j} R_{W_j} f}_{2}^2 \leq \max_j  ||H_{\om_j}||_{2 \rightarrow 2}^2  \sum_j \norm{R_{W_j} f}_{2}^2, \\
\norm{\max_j |H_{v_j} R_{W_j} f|}_{2}^2 &\leq \sum_j \norm{H_{v_j} R_{W_j} f}_{2}^2 \leq \max_j  ||H_{v_j}||_{2 \rightarrow 2}^2  \sum_j \norm{R_{W_j} f}_{2}^2.
\end{align*}
Since $\max_j ||H_{v_j}||_{2 \rightarrow 2} = 1$, the desired estimate \eqref{almost-ortho} will follow if we are able to show that 
\begin{equation} \label{RW-E}
\sum_j ||R_{W_j}f||_2^2 \leq ||E||_{L^\infty(\mathbb S^n)} ||f||_2^2, 
\end{equation} 
where $E$ is as in the statement of Theorem \ref{thm:ortho}.
\vskip0.1in 
\noindent We set about proving \eqref{RW-E}. By Plancherel's theorem, 
\begin{equation}  \sum_j \norm{R_{W_j} f}_{L^2}^2 = \sum_j \norm{\ind_{W_j} \widehat{f}}_{L^2}^2 \leq \norm{\sum_j \ind_{W_j} }_{L^\infty(\mathbb S^{n})} \norm{f}_{L^2}^2, \label{Plancherel} \end{equation} 
where we have used the fact that $\ind_{W_j}$ is homogeneous of degree 0 in the last inequality. Our claim is that for every $u \in \mathbb S^n$
\begin{align}\label{eqn:E}
\sum_j \mathbf 1_{W_j}(u) = \#\{ j : u \in W_j\} &\leq E(u), \text{ as a result of which }  \\
\bigl|\bigl| \sum_j \ind_{W_j} \bigr|\bigr|_{L^\infty(\mathbb S^{n})} &\leq \norm{E}_{L^\infty(\mathbb S^{n})}. \label{sup-norm-bounds}
\end{align}
Indeed, suppose that $u \in W_j\cap \mathbb S^n$ for some $j$. From the definition of $W_j$, it follows that $u \in \vec{v}^\perp$ for some $v\in \mathfrak L_j$, i.e., $ u\cdot \vec{v} = P_u(v) =0$. Thus $Z(P_u)$ intersects the line segment $\mathfrak L_j$ defined by \eqref{def-Lj}. By Lemma \ref{lem:inters}, we conclude that $Z(P_u)$ intersects $O_j$ as well. This implies that \[ \{j : u \in W_j \} \subseteq \{j : Z(P_u) \cap O_j \ne \emptyset \}, \] which leads to the estimate claimed in \eqref{eqn:E}. 
Substituting \eqref{sup-norm-bounds} into \eqref{Plancherel} yields \eqref{RW-E}, completing the proof. 
\end{proof} 

\section{Direction sets given by Cartesian products: Proof of Theorem \ref{thm:prod}}  \label{sec:prod}
\noindent Theorem \ref{thm:prod} consists of two statements. Section \ref{proof-prod-norm} below gives the proof of the upper bound \eqref{eqn:prod:norm}. Section \ref{sharpness-prod-section} establishes the sharpness of this bound.    
\subsection{Proof of \eqref{eqn:prod:norm}} \label{proof-prod-norm} 
We use a refinement of the idea that appeared in the new proof of Theorem \ref{Kim-theorem}. The proof is based on induction on $m$, where $1 \leq m \leq n$ is an index such that $N_k = 1$ for $k > m$. The case $m = n$ corresponds to the situation when $N_n > 1$.  
\vskip0.1in
\noindent Let us start with the base case $m=1$. This corresponds to the estimate 
\begin{equation}  \mathfrak C_{\text{prod}}(N_1, 1, \cdots, 1; n) \leq C \log(N_1+1). \label{base-case-m1}  \end{equation} 
For $n=1$, this is the well-known estimate \eqref{eqn:2d}. For $n > 1$, we invoke a standard slicing argument that has been proved in Lemma \ref{lem:slice}, with $n$ and $l$ in that lemma replaced by 1 and $(n-1)$ respectively. This shows that \[ \mathfrak C_{\text{prod}}(N_1, 1, \cdots, 1; n) = \mathfrak C_{\text{prod}}(N_1; 1), \] completing the verification of the base case.  
\vskip0.1in
\noindent We turn now to the induction step. This will be handled using Theorem \ref{thm:ortho}. 
Suppose that \eqref{eqn:prod:norm} has been proved (with $N_j$ replaced by $N_j'$) for every choice of integer vector $(N_1', \cdots, N_n')$ with $N_1' \geq N_2' \geq \cdots \geq N_n' \geq 1$, for which $N_k' = 1$ when $k > m-1$. The aim is to prove \eqref{eqn:prod:norm} for an integer sequence $N_1 \geq N_2 \geq \cdots N_n$ where $N_k = 1$ for $k > m$. 
\vskip0.1in
\noindent For each $1\leq k \leq m$, let $Q_{k,m}\geq 1$ be the integer part of the fraction $N_k/N_m$. Dictated by the definition \eqref{def-Cprod} of $\mathfrak C_{\text{prod}}$, we fix a direction set \[ \Omega = \prod_{k=1}^{n} \Omega_k, \quad \text{ with } \Omega_k \subseteq \mathbb R, \; \#(\Omega_k) = N_k  \text{ for $1 \leq k \leq n$}. \] Thus, for $k > m$, the set $\Omega_k$ is a singleton. For each index $k \leq m$, we order the elements of $\Omega_k \subseteq \mathbb R$ in increasing order, and pick $N_m$ disjoint open intervals $\{I(k, \ell_k): 1 \leq \ell_k \leq N_m\}$ in $\mathbb R$ of the form 
\begin{equation} 
I(k, \ell_k) = \bigl(\alpha_k(\ell_k -1), \alpha_k(\ell_k) \bigr), \quad \text{ with }  \alpha_{k}(0) = -\infty, \quad \alpha_{k}(N_m) = \infty, 
\end{equation} 
such that each $I(k, \ell_k)$ contains either $Q_{k,m}$ or $Q_{k,m}+1$ consecutive elements of $\Omega_k$. For $\vec{\ell} = (\ell_1, \cdots, \ell_m)$, we define an $m$-dimensional rectangular parallelepiped $O_{\vec{\ell}}$ as follows,
\[ O_{\vec{\ell}} := \prod_{k=1}^{m} I(k, \ell_k) \times \prod_{k = m+1}^{n} \Omega_k, \quad \text{ and set } \quad  \Omega_{\vec{\ell}} := \Omega \cap O_{\vec{\ell}}.  \]
The sets $O_{\vec{\ell}}$ are connected, and form a finite cover of $\Omega$; they will serve as the sets $O_j$ required in Theorem \ref{thm:ortho}. 
For this choice of sets, and in the notation of Theorem \ref{thm:ortho}, we claim the following: 
\begin{align}
&{\text{For every $u \in \mathbb S^n$}}, E(u) := \# \bigl\{\vec{\ell} : O_{\vec{\ell}} \cap Z(P_u) \ne \emptyset  \bigr\}  \leq C (mN_m)^{m-1}, \label{bound-on-E} \\
&\max_{\vec{\ell}} ||H_{\Omega(\vec{\ell})}||_{2 \rightarrow 2} \leq C_0 \Bigl[ \prod_{k=2}^{m-1} \frac{N_k}{N_m} \Bigr]^{\frac{1}{2}} \log \left( \frac{N_1}{N_2} + 1 \right),  \text{ and } \label{H-om-l}  \\
&\exists \text{ a set } \mathcal O \text{ with }  \#(\mathcal O \cap O_{\vec{\ell}}) = 1 \text{ for every } \vec{\ell}, \; ||H_{\mathcal O}||_{2 \rightarrow 2} \leq C N_m^{\frac{m-1}{2}}. \label{H-O} 
 \end{align} 
Assuming these estimates for now, we substitute them  into \eqref{almost-ortho} in Theorem \ref{thm:ortho} to obtain 
\begin{align*}
||H_{\Omega}||_{2 \rightarrow 2} &\leq ||H_{\mathcal O}||_{2 \rightarrow 2} + ||E||_{L^\infty(\mathbb S^n)}^{\frac12} \Bigl( \max_{\vec{\ell}} ||H_{\Omega_{\vec{\ell}}}||_{2 \rightarrow 2} + 1 \Bigr) \\ 
&\leq CN_m^{\frac{m-1}{2}} + C_0 \bigl[ C(mN_m)^{m-1} \bigr]^{\frac{1}{2}} \Bigl[ \prod_{k=2}^{m-1} \frac{N_k}{N_m} \Bigr]^{\frac{1}{2}} \log \left( \frac{N_1}{N_2} + 1 \right) \\
&\leq C_n N_m^{\frac{m-1}{2}} \Bigl[ \prod_{k=2}^{m-1} \frac{N_k}{N_m} \Bigr]^{\frac{1}{2}} \log \left( \frac{N_1}{N_2} + 1 \right) \\ 
&= C_n  \Bigl[ \prod_{k=2}^{m} N_k \Bigr]^{\frac{1}{2}} \log \left( \frac{N_1}{N_2} + 1 \right). 
\end{align*}
This completes the induction, and hence the proof of \eqref{eqn:prod:norm} up to the verification of the claims \eqref{bound-on-E}, \eqref{H-om-l} and \eqref{H-O}. We now turn to the proof of these claims.

\begin{lem} \label{bound-on-E-lemma}
In the setup described in section \ref{proof-prod-norm}, the estimate \eqref{bound-on-E} holds. 
\end{lem} 
\begin{proof} 
As in the proof of Theorem \ref{Kim-theorem}, we reduce the estimation to a counting problem involving a polynomial zero set.  Let us first identify the sets $\mathcal O_{\vec{\ell}}$ as the connected components of $V(\mathbb R) \setminus Z(P)$, for an $m$-dimensional algebraic variety $V$ and an appropriately defined polynomial $P$. We choose 
\begin{align*}  &V = \mathbb C^m \times \prod_{k = m+1}^{n} \Omega_k, \quad \text{ so that } \quad V(\mathbb R) = \mathbb R^m \times \prod_{k = m+1}^{n} \Omega_k, \text{ and }  \\ 
&P(x_1, \cdots, x_n) = \prod_{k=1}^{m} \prod_{\ell_k = 1}^{N_m-1} \bigl(x_k - \alpha_k(\ell_k) \bigr), \; \text{ so that } \; \deg(P) = m(N_m-1).
\end{align*} 
Then $V$ is an algebraic variety in $\C^n$ of dimension $m$ and degree 1, given by the zero set of finitely many linear polynomials $\{P_k : m+1 \leq k \leq n\}$, where 
\[ P_{k}(x) = x_k - \omega_k \; \text{ with }  \Omega_k = \{\omega_k\}. \]  The set $Z(P)$ consists of a union of coordinate hyperplanes of the form $x_k = \alpha_k(\ell_k)$, with $1 \leq k \leq m$ and $1 \leq \ell_k < N_m$. These hyperplanes partition $V(\mathbb R) = \mathbb R^m$ into the cells $O_{\vec{\ell}}$. The quantity $E(u)$ is then the number of connected components of $V(\mathbb R) \setminus Z(P)$ that intersect $Z(P_u)$.  Proposition \ref{lem:boundE} offers a general bound for this quantity that shows that in this case $E(u)$ is bounded by a constant multiple of $ [\deg(P)]^{m-1} < (mN_m)^{m-1}$, as claimed. 
\end{proof}
\begin{lem} \label{H-om-l-lemma} 
In the setup described in section \ref{proof-prod-norm}, the estimate \eqref{H-om-l} holds. 
\end{lem} 
\begin{proof} 
For each multi-index $\vec{\ell} = (\ell_1, \cdots, \ell_m)$,  
\[ \Omega (\vec{\ell})  :=  \Omega \cap O_{\vec{\ell}} = \prod_{k=1}^n \Omega_k (\vec{\ell}),  \quad \text{ where } \quad \Omega_k(\vec{\ell}) :=  \begin{cases} \Omega_k \cap I(k, \ell_k) &\text{ if } k \leq m,  \\ \Omega_k &\text{ if } k > m. \end{cases} \]
In other words, each direction set $\Omega(\vec{\ell})$ is of product type, with $\#(\Omega_k(\vec{\ell})) \leq Q_{k,m}+1$ for $k < m$ and $\#(\Omega_{k}(\vec{\ell}) )= 1$ for $k \geq m$.  Further, the hypothesis $N_1 \geq N_2 \geq \cdots \geq N_m$ implies that $Q_{1,m} \geq Q_{2,m} \geq \cdots \geq Q_{m-1,m}$. Thus the induction hypothesis \eqref{eqn:prod:norm} applies, with $m$ replaced by $m-1$, to $\Omega(\vec{\ell})$ for each multi-index $\vec{\ell}$, and yields
\begin{align*} 
\max_{\vec{\ell}} ||H_{\Omega(\vec{\ell})}||_{2 \rightarrow 2} &\leq \mathfrak C_{\text{prod}}(Q_{1,m}+1, \cdots, Q_{m-1,m}+1, 1, \cdots, 1; n) \\ 
&\leq C \Bigl[ \prod_{k=2}^{m-1} (Q_{k,m}+1)\Bigr]^{\frac{1}{2}} \log \left( \frac{Q_{1,m}+1}{Q_{2,m}+1} + 1 \right)  \\ 
&\leq C_0 \Bigl[ \prod_{k=2}^{m-1} \frac{N_k}{N_m} \Bigr]^{\frac{1}{2}} \log \left( \frac{N_1}{N_2} + 1 \right).  
\end{align*}     
This is the claimed estimate \eqref{H-om-l}.
\end{proof} 
\begin{lem} \label{H-O-lemma}
In the setup described in section \ref{proof-prod-norm}, the claim in \eqref{H-O} holds. 
\end{lem} 
\begin{proof} 
Let us first describe the set $\mathcal O$. For each $k \leq m$ and $1 \leq \ell_k \leq N_m$, we pick a point $a(k, \ell_k) \in I(k, \ell_k)$. Define
\[ \mathcal O_k :=  \left\{ \begin{aligned} &\bigl\{a(k, \ell_k) : 1 \leq \ell_k \leq N_m \bigr\} &\text{ if } k \leq m, \\ &\Omega_k &\text{ if } k > m, \end{aligned} \right\} \quad \text{ and } \quad \mathcal O := \prod_{k=1}^{n} \mathcal O_k.  \]  
Clearly, $\mathcal O$ is of product type, with $\#(\mathcal O_k) = N_m$ for all $k \leq m$, and $\#(\mathcal O_k) = 1$ for all $k > m$. Moreover, for every multi-index $\vec{\ell} = (\ell_1, \cdots, \ell_m)$, 
\[ O_{\vec{\ell}} \cap \mathcal O = \Bigl[ \prod_{k=1}^{m} I(k, \ell_k) \cap \mathcal O_k \Bigr] \times \prod_{k=m+1}^n \Omega_k = \Bigl[ \prod_{k=1}^{m} \{a(k, \ell_k)\} \Bigr] \times \prod_{k=m+1}^n \Omega_k,  \]
which is a single point in $\mathbb R^n$.
\vskip0.1in  
\noindent Our next task is to estimate the $L^2$ operator norm of $H_{\mathcal O}$. By Lemma \ref{lem:slice}, 
\[ ||H_{\mathcal O}||_{L^2(\mathbb R^{n+1}) \rightarrow L^2(\mathbb R^{n+1})} = || H_{\mathcal O^{\ast}} ||_{L^2(\mathbb R^{m+1}) \rightarrow L^2(\mathbb R^{m+1})}, \; \text{ where } \; \mathcal O^{\ast} = \prod_{k=1}^{m} \mathcal O_k.\] 
The direction set $\mathcal O^{\ast}$ is an $m$-fold Cartesian product of sets with equal cardinalities; so the result of \cite{Kim}, as given in \eqref{Kim-Cprod}, applies to it. Since $m \geq 2$, invoking this result yields  
\begin{equation} 
||H_{\mathcal O}||_{2 \rightarrow 2} = ||H_{\mathcal O^{\ast}}||_{2 \rightarrow 2} \leq \mathfrak C_{\text{prod}}(N_m, \cdots, N_m; m) \leq C N_m^{\frac{m-1}{2}}. 
\end{equation}      
This concludes the proof of \eqref{H-O} 
\end{proof}

\subsection{Sharpness of \eqref{eqn:prod:norm}} \label{sharpness-prod-section} 
For a sequence of integers $N_1 \geq \cdots \geq N_n \geq 1$, we set $N = N_1 N_2 \cdots N_n$ and
\begin{equation} \label{sharpness-Omega}  \quad \om = \prod_{k=1}^n U_{N_k}, \; \text{ where } \; U_{N_k} = \bigl\{ j/N_k : 1\leq j \leq N_k \bigr\}. \end{equation}  
Thus $\#(\Omega) = N$. The goal is to show that 
\begin{equation} 
||H_{\Omega}||_{L^2(\mathbb R^{n+1}) \rightarrow L^2(\mathbb R^{n+1})} \geq c \Bigl(\frac{N}{N_1} \Bigr)^{\frac12} \log \Bigl( \frac{N_1}{N_2}+1\Bigr), \label{sharpness-product} 
\end{equation} 
for some constant $c > 0$ that is independent of $N_1, N_2, \cdots, N_k$.  
\vskip0.1in
\noindent We choose the test function 
\begin{equation}  f(y_1,\cdots,y_{n+1}) = \frac{\ind_{\mathcal R} (y_1,\cdots,y_{n+1})}{N_2 N_1^{-1} + y_1 + y_{n+1}}, \; \mathcal R = [0, 5] \times [0, 2]^{n-1} \times [0,1].  \label{test-function-f} \end{equation} 

Concerning this test function $f$, we claim the followings. 
\begin{lem} \label{f-L2-norm-lemma} 
For $f$ as in \eqref{test-function-f}, we have the estimate 
\[ \norm{f}_{2}^2 \sim \log \Bigl(\frac{N_1}{N_2} + 1 \Bigr). \]
\end{lem}
\begin{proof} 
A direct calculation shows that
\begin{align*}
||f||_2^2 = 2^{n-1} \left( \log \Bigl(  1 + \frac{N_1}{N_2} \Bigr)  - \log \Bigl( \frac{ N_2N_1^{-1} + 6}{ N_2N_1^{-1} + 5} \Bigr) \right). 
\end{align*} 
Since $0< N_2N_1^{-1}  \leq 1$, 
\[ \log \Bigl(  1 + \frac{N_1}{N_2} \Bigr)  > 2 \log \Bigl( \frac{ N_2N_1^{-1} + 6}{ N_2N_1^{-1} + 5} \Bigr). \]
Therefore, we get the bound 
\[  2^{n-2}  \log \Bigl( 1 + \frac{N_1}{N_2}\Bigr)\leq   ||f||_2^2 \leq 2^{n-1} \log \Bigl( 1 + \frac{N_1}{N_2}\Bigr).  \]
\end{proof} 

\begin{lem} \label{Homf-L2-lemma} 
Let $f$ be as in \eqref{test-function-f}. There exist a constant $c>0$ and a collection of sets $\{S_v : v \in \Omega \}$ in $\mathbb R^{n+1}$ such that 
\begin{align}
S_v \cap S_{v'} &= \emptyset \text{ for } v \ne v', \text{ and } \label{Sv-disjoint} \\
|| \bigl[H_{v} f \bigr]  \ind_{S_v}||_2^2 &\geq \frac{c}{N_1} \Bigl[ \log \Bigl( \frac{N_1}{N_2} + 1\Bigr) \Bigr]^3 \text{ for every } v \in \Omega.  \label{Homf-L2-norm} 
\end{align} 
\end{lem} 

Before we give a proof of Lemma \ref{Homf-L2-lemma}, we proceed to prove our claim \eqref{sharpness-product}. 
Since $|H_{\Omega}f|$ pointwise dominates $|H_{v}f|$ for every $v \in \Omega$, the disjointness of the sets $S_v$ (as given by \eqref{Sv-disjoint}) leads to the following estimate: 
\begin{align}  ||H_{\Omega} f||_2^2 &\geq \sum_{v \in \Omega} || \bigl[ H_{v} f \bigr] \ind_{S_v}||_2^2 \nonumber \\ 
&\geq \sum_{v \in \Omega} \frac{c}{N_1} \Bigl[ \log \Bigl( \frac{N_1}{N_2} + 1\Bigr) \Bigr]^3 \geq c \frac{N}{N_1} \Bigl[ \log \Bigl( \frac{N_1}{N_2} + 1\Bigr) \Bigr]^3. \label{Hom-L2-lower-bound} 
\end{align} 
In the second line of the displayed sequence above, we have substituted \eqref{Homf-L2-norm} into the right hand side of the previous expression.  Combining \eqref{Hom-L2-lower-bound} with Lemma \ref{f-L2-norm-lemma}, we arrive at  
\[ ||H_{\Omega}||_{2 \rightarrow 2}^2 \geq  \frac{||H_{\Omega}f||_2^2}{||f||_2^2} \geq c \frac{N}{N_1} \Bigl[ \log \Bigl( \frac{N_1}{N_2} + 1\Bigr) \Bigr]^2,\]
which is the desired estimate \eqref{sharpness-product}.  
\vskip0.1in
\noindent It only remains to verify Lemma \ref{Homf-L2-lemma}.
\subsubsection{Proof of Lemma \ref{Homf-L2-lemma}}  
\begin{proof} 
Let $\Omega$ be as in \eqref{sharpness-Omega}. For each $v=(v_1,\cdots,v_n) \in \om$, let us set
\begin{equation} \label{def-Wv} W_v := \left\{ x\in \R_{-}^{n+1} \Bigl| \; \; \begin{aligned} &2N_2 < |x_{n+1}| < 4N_1, \\  &-{1}/{N_k} < \frac{x_k}{x_{n+1}} - v_k< 0\; \text{ for } 1\leq k \leq n. \end{aligned} \right\}, \end{equation} 
where $\R_{-}$ denotes $(-\infty, 0)$.
We first argue that the sets $W_v$ are disjoint. Indeed, if $v =(v_1, \cdots, v_n)$ and $v' = (v_1', \cdots, v_n')$ are two distinct elements of $\Omega$, then there exists an index $k \in \{1, \cdots, n\}$ such that $v_k \ne v_k'$. The definition \eqref{sharpness-Omega} of $\Omega$ implies that $v_k = j_k/N_k$ and $v_k' = j_k'/N_k$ for integers $j_k, j_k' \in \{1, \cdots, N_k\}$, $j_k \ne j_k'$. Thus, 
\[ \bigcap_{v \in \{v_k, v_k'\}} \Bigl(v - \frac{1}{N_k}, v \Bigr) = \bigcap_{j \in \{ j_k, j_k' \}} \Bigl(\frac{j-1}{N_k}, \frac{j}{N_k} \Bigr) = \emptyset. \] 
In other words, we have $\pi_k(W_v) \cap \pi_k(W_{v'}) = \emptyset$, where $\pi_k(x) = x_k/x_{n+1}$. Thus $W_v \cap W_{v'} = \emptyset$, proving the claimed disjointness. The set $S_v$ mentioned in Lemma \ref{f-L2-norm-lemma} will be a suitably chosen subset of $W_v$, thus ensuring \eqref{Sv-disjoint}. The remainder of the proof is devoted to verifying \eqref{Homf-L2-norm}, for $f$ as in \eqref{test-function-f}. We will do this by establishing an explicit pointwise lower bound on $H_vf$ on a subset of $W_v$.   
\vskip0.1in
\noindent A consequence of the definition \eqref{def-Wv} is that if $x \in W_v$, the (apriori signed) integral defining $H_vf(x)$ becomes sign-specific. More precisely, the integrand in
\[ H_vf(x) = \int \frac {\ind_{\mathcal R}(x - \langle v, 1 \rangle t)}{N_2 N_1^{-1} + (x_1 - v_1t) + (x_{n+1} - t)} \frac{dt}{t} \]   
is non-zero only if $t < 0$. In other words, by replacing $t$ by $-t$ we obtain
\begin{equation} \label{Hv-t}
\bigl| H_vf(x) \bigr| = \int_{0}^{\infty} \frac {\ind_{\mathcal R}(x + \langle v, 1 \rangle t)}{N_2 N_1^{-1} + (x_1 + v_1t) + (x_{n+1} + t)} \frac{dt}{t} 
\end{equation} 
In the next few steps, we will sequentially identify subsets of $W_v$ for which the integral above can be further simplified, eventually reducing it to a form  that can be directly integrated. To this end, we introduce an auxiliary set
\begin{align} \label{def-Xv-0} 
X_v &:= \bigcap_{k=1}^{n}\Bigl\{ x \in W_v : \frac{a_k -v_k}{x_{n+1}} < \frac{x_k}{x_{n+1}} - v_k < 0 \Bigr\}, \; \text{where } \\ a_k &= \begin{cases} 5 &\text{ if } k = 1, \\ 2 &\text{ if }  k \geq 2. \end{cases}  \nonumber 
\end{align} 
\noindent We first verify that $X_v$ is nonempty. In fact, for $2N_2 \leq |x_{n+1}| \leq 4N_1$, the quantity $(a_k - v_k)/|x_{n+1}|$ is $< 1/N_k$ for $k \geq 2$,  and is $> 1/N_k$ for $k =1$. 
Hence $X_v$ admits the alternative description
\begin{equation} \label{def-Xv}
X_v := \left\{ x\in \R_{-}^{n+1} \Biggl| \; \; \begin{aligned} &2N_2 < |x_{n+1}| < 4N_1, \\ &0 < x_1 - v_1 x_{n+1} < \frac{|x_{n+1}|}{N_1}, \\  & 0 < x_k - v_k x_{n+1} < a_k - v_k \; \text{ for } 2\leq k \leq n. \end{aligned} \right\} 
\end{equation} 
 The relevance of the set $X_v$ is that for every $x \in X_v$,  
\begin{equation} \label{geometric-fact-2} \bigl\{t \in \mathbb R : x + \langle v, 1 \rangle t \in \mathcal R \bigr\}= I_0, \end{equation} 
where we write, for $0\leq b<1$,
\begin{equation}
I_b := \bigl\{t \in \mathbb R: b < x_{n+1} + t < 1 \bigr\}.
\end{equation}
We will prove this geometric fact in a moment. Assuming this for now, we see that 
\begin{align} 
\bigl| H_vf(x) \bigr| &= \int_{I_0} \frac{1}{N_2 N_1^{-1} + (x_1 + v_1t) + (x_{n+1} + t)} \frac{dt}{t} \\ 
&\geq \frac{1}{2|x_{n+1}|}  \int_{I_0} \frac{1}{N_2 N_1^{-1} + (x_1 + v_1t) + (x_{n+1} + t)} \, dt.  \label{Hv-second-round}
\end{align} 
The last inequality uses the fact that $0<t < 1 - x_{n+1} = 1 + |x_{n+1}| < 2|x_{n+1}|$ if $t\in I_0$. 
\vskip0.1in
\noindent Next we restrict the range of $t$ further, in order to remove the dependence of the integrand on $x_1$. To do so, we note that 
\begin{equation}  \label{x1-xn+1-ineq} 5(x_{n+1} + t) \geq (x_1 + v_1 t) \quad \text{ if and only if } \quad x_{n+1} + t \geq \frac{x_1 - v_1 x_{n+1}}{5 - v_1}.\end{equation} 
A re-arrangement of the first defining inequality of $X_v$ in \eqref{def-Xv-0} (involving the variable $x_1$) yields 
\begin{equation}  0 < b(x) := \frac{x_1 - v_1 x_{n+1}}{5 - v_1} < 1. \label{def-a} \end{equation} Using the relation \eqref{x1-xn+1-ineq} therefore leads to the following estimate on the integral in \eqref{Hv-second-round}:
\begin{align}
\bigl| H_vf(x) \bigr|  &\geq \frac{1}{2|x_{n+1}|}\int_{I_{b(x)}} \frac{1}{N_2 N_1^{-1} + 6(x_{n+1} + t)} dt, \; \text{ which in turn is } \nonumber \\ 
&\geq \frac{c}{|x_{n+1}|} \int_{I_{b(x)}} \frac{1}{x_{n+1} + t} dt = \frac{c}{|x_{n+1}|} \log(1/b(x)), \label{Hv-third-round} 
\end{align}   
provided $5 \cdot b(x) \geq N_2/N_1$ for some absolute constant $0<c<1$. 
\vskip0.1in
\noindent This last requirement leads to the definition of $S_v$:  
\[ S_v := \Bigl\{ x \in X_v : x_1 - v_1 x_{n+1} \geq \frac{N_2}{N_1}\Bigr\}, \] or written explicitly from \eqref{def-Xv},
\begin{equation} 
S_v := \left\{ x \in \R_{-}^{n+1} \Biggl| \; \; \begin{aligned} &\frac{N_2}{N_1} < x_1 - v_1 x_{n+1} < \frac{|x_{n+1}|}{N_1}, \\ &2N_2 < |x_{n+1}| < 4N_1, \\ &0 < x_k - v_k x_{n+1} < 2-v_k \text{ for } 2 \leq k \leq n. \end{aligned} \right\}  \label{def-Sv}
\end{equation} 
The computations leading up to \eqref{Hv-third-round} and \eqref{def-Sv} show that for $x \in S_v$, 
\begin{align}  
\bigl| H_vf(x) \bigr| &\geq \frac{c}{|x_{n+1}|} \log \Bigl(\frac{5- v_1}{x_1 - v_1 x_{n+1}} \Bigr) \nonumber  \\
& \geq \frac{c}{|x_{n+1}|} \log \Bigl( \frac{4N_1}{|x_{n+1}|}\Bigr). \label{Hv-final} \end{align}  
We will use this estimate to arrive at \eqref{Homf-L2-norm}. 
\vskip0.1in
\noindent In preparation for computing the $L^2$ norm of $H_vf$ on $S_v$, and in view of the representation \eqref{def-Sv}, we make a change of variables in $S_v$, setting \[ z_k = x_k - x_{n+1}v_k \; \text{ for } 1 \leq k \leq n, \quad \text{ and }  \quad z_{n+1} = -x_{n+1}. \]  Incorporating this into \eqref{Hv-final} results in the following estimate: 
\begin{align*}
\int_{S_v} \bigl| H_v f(x) \bigr|^2 &\geq c \prod_{k=2}^n (2-v_k) \int_{2N_2}^{4N_1} \frac{1}{z_{n+1}^2} \Bigl[\log \Bigl(\frac{4N_1}{z_{n+1}} \Bigr) \Bigr]^2 \int_{ N_2/N_1}^{z_{n+1}/N_1}\, dz_1  \, dz_{n+1} \\ 
&\geq \frac{c}{N_1} \int_{2N_2}^{4N_1} \frac{(z_{n+1} - N_2)}{ z_{n+1}^2} \Bigl[\log \Bigl(\frac{4N_1}{z_{n+1}} \Bigr) \Bigr]^2 \, dz_{n+1} \\ 
 &\geq \frac{c}{2N_1} \int_{2N_2}^{4N_1}  \Bigl[\log \Bigl(\frac{4N_1}{z_{n+1}} \Bigr) \Bigr]^2 \frac{dz_{n+1}}{z_{n+1}} = \frac{c}{6N_1} \Bigl[\log \Bigl( \frac{2N_1}{N_2}\Bigr) \Bigr]^3 \\ 
 &\geq \frac{c}{6N_1} \Bigl[ \log \left( \frac{N_1}{N_2} + 1\right) \Bigr]^3. 
\end{align*} 
This is the estimate claimed in \eqref{Homf-L2-norm}, which completes the proof. 
\end{proof} 

\subsubsection{Proof of \eqref{geometric-fact-2}}
\begin{proof} 
Since the inclusion $\subset$ is trivial, it suffices to prove the inclusion $\supset$.
Suppose that $x \in X_v$, with $X_v$ as in \eqref{def-Xv}. Suppose also that $0 < x_{n+1} + t < 1$. Then $x + \langle v, 1\rangle t \in \mathcal R$ if and only if $0 < x_k + v_k t < a_k$ for $1 \leq k \leq n$. Here $a_k$ is the constant defined in \eqref{def-Xv-0}. Accordingly, we check 
\[ x_k + v_k t = (x_k - v_k x_{n+1}) + v_k ( x_{n+1} + t) \] 
In view of the defining inequalities for $X_v$ given in \eqref{def-Xv}, the first term in parentheses above is bounded below and above by 0 and $(a_k - v_k)$ respectively.  The condition $0 < x_{n+1} + t < 1$ says that the second term is bounded between 0 and $v_k$. Adding the two terms therefore results in the desired inequality.    
\end{proof} 
\section{Improved estimates in $\mathbb R^3$, Part 1: Proof of Theorem \ref{thm:3d}} \label{proof:thm3d:section}  
\subsection{Ingredients of the proof} In addition to Theorem \ref{thm:ortho}, the proof of Theorem \ref{thm:3d} relies on two facts. The first is a polynomial partitioning result due to Guth and Katz \cite{GK} . 
\begin{theorem}[{\cite[Theorem 4.1]{GK}}] \label{thm:GK}
There exists an absolute constant $A_1\geq 2$, depending only on $n$, with the following property. Given any integer $D\geq 1$ and any finite set $\om\subset \R^{n}$ of cardinality $N$, there is an $n$-variate polynomial $P \in \mathbb R[x_1, \cdots, x_n]$ that is not identically zero and has degree at most $D$, so that $\R^{n} \setminus Z_{\mathbb R}(P)$ is a disjoint union of at most $A_1 D^n$ open connected components $O_j$, each containing at most $A_1 N {D}^{-n}$ elements of $\om$. 
\end{theorem}
\noindent In the above statement, the fact that $\R^{n} \setminus Z_{\mathbb R}(P)$ has at most $O(D^n)$ connected components is due to Milnor \cite{Mil} and Thom \cite{Thom}. See also \cite[Theorem A.2]{ST} for a generalization of the Milnor-Thom bound. 
\vskip0.1in
\noindent The second ingredient of the proof is a recursion inequality in the spirit of \eqref{eqn:indcurve}. 
\begin{prop}\label{lem:3d}
There exists an absolute constant $0<c<1$ such that for every $d \geq1$ and $N \geq c^{-1} d^2$,  
\begin{equation}  \mathfrak C_2^{\ast}(N;d) \leq \mathfrak C_2^{\ast}(N/2; d) + 5 \sqrt{d}.  \label{C2-star-recursion} \end{equation}
\end{prop}
\begin{proof} 
The proof relies on the structure of the zero set of a bivariate polynomial. 
In Lemma \ref{lem:decomposition} of the appendix, we will show that 
there exists an absolute constant $A_2 > 0$ such that for any $d \geq 1$ and any nontrivial bivariate polynomial $P$ of degree at most $d$,  
we can write $Z_{\mathbb R}(P)$, in suitable coordinates, as a disjoint union of at most $A_2 d^2$ points and $A_2 d^2$ curves, where each curve is given by a graph of the form $\{(x, g(x)) : x \in I\}$ for some continuous function $g:I \rightarrow \mathbb R$ and some interval $I \subseteq \mathbb R$.  
\vskip0.1in
\noindent Suppose now that $\Omega \subseteq Z_{\mathbb R}(P)$ for some polynomial $P \in \mathcal P_d(2)$, $\#(\Omega) = N$. The operator norm $||H_{\Omega}||_{2 \rightarrow 2}$ is invariant under affine coordinate transformations of $\Omega$ and hence of $Z_{\mathbb R}(P)$. Therefore, choosing an appropriate set of coordinates and using Lemma \ref{lem:decomposition}, we write $Z_{\mathbb R}(P)$ as the disjoint union of its connected components $Z_{\ell}$: 
\begin{align*} 
&Z_{\mathbb R}(P) = Z_{\mathbb R}(P; \text{points}) \bigsqcup Z_{\mathbb R}(P; \text{curves}), \quad \text{ where } \\ 
&Z_{\mathbb R}(P; \text{points}) \text{ is a union of points with } \# \bigl(Z_{\mathbb R}(P; \text{points}) \bigr) \leq A_2 d^2 \quad \text{ and } \\
&Z_{\mathbb R}(P; \text{curves}) = \bigsqcup_{\ell = 1}^{\ell_0} Z_{\ell}, \;\; \ell_0 \leq A_2 d^2, \text{ where each $Z_{\ell}$ is the graph of a curve.}
\end{align*}  
For $1 \leq \ell \leq \ell_0$, suppose that $Z_{\ell} = \{(x, g_{\ell}(x)) : x \in I_{\ell} \}$, for some interval $I_{\ell} \subset \mathbb R$ and some continuous function $g_{\ell} : I_{\ell} \rightarrow \mathbb R$. Let $\pi$ denote the projection of $Z_{\ell}$ onto the horizontal axis. Then $\pi: Z_{\ell} \rightarrow I_{\ell}$ is a continuous bijection with a continuous inverse. Further, if $\Omega \cap Z_{\ell} \ne \emptyset$, then $\pi(\Omega \cap Z_{\ell})$  is a non-empty finite subset of $I_{\ell}$.  We decompose $I_{\ell}$ into the smallest number of disjoint subintervals $\{I_{\ell r} : r \geq 1\}$ such that each $I_{\ell r}$ contains at least one and no more than 4 points of $\pi(\Omega \cap Z_{\ell})$. Since connectedness is preserved under continuous maps, the pull-back of the projection $\pi$ generates a partition of $Z_{\ell}$ into disjoint connected subsets $\{ Z_{\ell r}  = \pi^{-1}(I_{\ell r}): r \geq 1\}$, such that each $Z_{\ell r}$ contains at least one and no more than 4 points of $\Omega$. For a given index $\ell$, the number of such connected sets $Z_{\ell r}$ contained in $Z_{\ell}$ is exactly $\lceil \#(\Omega \cap Z_{\ell})/4 \rceil$, hence at most  $\#(\Omega \cap Z_{\ell})/4 + 1$. 
\vskip0.1in
\noindent In order to apply Theorem \ref{thm:ortho}, we still need to define the various quantities required by the theorem. The collection of connected components $\{O_j\}$ will consist of 
the isolated points in $Z_{\mathbb R}(P)$, and the pieces $Z_{\ell r}$ of the curves $Z_{\ell}$ mentioned above. In other words, a set $O_j$ can be of two types: either $O_j = \{x_0\}$ for some $x_ 0 \in Z_{\mathbb R}(P; \text{points})$,  or $O_j = Z_{\ell r}$ for some $\ell$ and $r$. 
Clearly, the sets $O_j$ form a finite cover of $\Omega$. We pick a single point from each set $\Omega_j = \Omega \cap O_j$ to create the set $\mathcal O$ specified in Theorem \ref{thm:ortho}. Then $\mathcal O \subseteq \Omega \subseteq Z_{\mathbb R}(P)$, with 
\begin{align}  \#(\mathcal O) &\leq  \#  \bigl(Z_{\mathbb R}(P; \text{points}) \bigr) + \sum_{\ell = 1}^{\ell_0} \Bigl[ \#(\Omega \cap Z_{\ell})/4 + 1 \Bigr] \\ &\leq A_2 d^2 + \frac{\#(\Omega)}{4} + \ell_0 \leq \frac{N}{4} + 2A_2 d^2 \leq \frac{N}{4} + \frac{N}{4} = \frac{N}{2}. \label{cardinality-O}
\end{align}  
At the last step, we have chosen the constant $c > 0$ small enough so that $2A_2 c < 1/4$, which implies that $2A_2 d^2 \leq N/4$ for $d^2 \leq cN$. The choice of $O_j$ also dictates that 
\begin{equation} \label{cardinality-Oj}
\#(\Omega_j) \leq 4 \text{ for all } j. 
\end{equation} 
\vskip0.1in
\noindent Before applying Theorem \ref{thm:ortho}, it remains to estimate $E(u)$, which represents the number of sets $O_j$ that intersect a line in $\mathbb R^2$ parametrized by $u$. Since each $O_j$ has been chosen to be a subset of $Z_{\mathbb R}(P)$, clearly $E(u)$ is dominated by the number of points of intersection between the line and $Z_{\mathbb R}(P)$. It suffices therefore to estimate this last quantity for a general $u$. First, we observe that, in view of the degree of $P$, the zero set $Z_{\mathbb R}(P)$ can contain at most $d$ lines; all other lines intersect $Z_{\mathbb R}(P)$ in at most $d$ points. Hence 
\begin{equation} \label{E-d}  
||E||_{\infty} \leq d.
\end{equation}    
\vskip0.1in
\noindent With \eqref{cardinality-O}, \eqref{cardinality-Oj} and \eqref{E-d} in place, we invoke Theorem \ref{thm:ortho} to obtain that 
\begin{align*} 
||H_{\Omega}||_{2 \rightarrow 2} &\leq ||H_{\mathcal O}||_{2 \rightarrow 2} + ||E||_{L^{\infty}(\mathbb S^2)}^{1/2} \Bigl( \max_j ||H_{\Omega_j}||_{2 \rightarrow 2} + 1\Bigr) \\ 
&\leq \mathfrak C_2^{\ast}(N/2, d) + \sqrt{d} (\sup_j \# \Omega_j + 1) \leq  \mathfrak C_2^{\ast}(N/2, d) + 5 \sqrt{d} . 
\end{align*}    
This gives the desired recursive inequality \eqref{C2-star-recursion}.
\end{proof} 
\noindent Given the two ingredients in this section, the proof of Theorem \ref{thm:3d} is completed as follows.

\subsection{Proof of Theorem \ref{thm:3d}} 
\begin{proof} 
The proof uses a two-tiered induction process involving the lexicographic ordering of the pair $(N, d)$. Specifically, we declare $(N', d') < (N, d)$ if either (a) $N' < N$ or (b) $N' = N$ and $d' < d$. The goal is to establish \eqref{Zp-3d} for $(N', d') = (N, d)$, assuming that it holds for all $(N', d') < (N, d)$.  
\vskip0.1in
\noindent 
As the base of the induction, we first verify that the statement \eqref{Zp-3d} is true when $N \leq  2 A_1 c_0^{-2}$ for any $d \geq 1$, where $A_1$ is the constant from Theorem \ref{thm:GK} and $c_0$ is a small absolute constant defined below in \eqref{choice-of-c0}. In this case, \eqref{Zp-3d} holds by the trivial bound $\mathfrak C_2^{\ast}(N;d) \leq N$, with any constant $A$ obeying 
\begin{equation} \label{choice-of-A-1} 
A \geq 2 A_1 c_0^{-2}.
\end{equation}  
\vskip0.1in
\noindent Let us proceed to the induction step. We assume that the estimate \eqref{Zp-3d} holds with some sufficiently large constant $A$ for all tuples $(N', d') < (N, d)$ for some $N > 2 A_1 c_0^{-2}$. 
We will prove that \eqref{Zp-3d} holds for $(N', d') = (N, d)$ with the same constant. As we will see, the constant $A$ will be chosen to depend only on the constants $A_1$ and $c$ that appear in Theorem \ref{thm:GK} and Proposition \ref{lem:3d}, respectively. We note that the size condition $N>2 A_1 c_0^{-2} $ ensures that 
\[ cN \gg 1,\;\; \text{ and } \;\; c_0 N^{1/2} \geq 2 . \]
\vskip0.1in
\noindent 
We split the inductive step into two cases, depending on the relative sizes of $N$ and $d$. In what follows, $c$ will refer to the constant from Proposition \ref{lem:3d}.
\vskip0.1in
\noindent {\em{Case 1 (small $d$) : }}  Suppose that $1\leq d^2 \leq cN$. By Proposition \ref{lem:3d}, 
\[ \mathfrak C_2^{\ast}(N;d) \leq \mathfrak C_2^{\ast}(N/2;d) + 5 \sqrt{d} \leq A\sqrt{d} \log{(N/2)} +5 \sqrt{d}.\]
The last expression in the display above follows from the induction hypothesis applied to $(N/2, d) < (N, d)$. It is bounded above by $A \sqrt{d} \log N$ provided  
\begin{equation} \label{A4} 
A \geq 5/\log 2, 
\end{equation} 
and the induction closes in this case. 
\vskip0.1in
\noindent {\em{Case 2 (large $d$) :}} Next suppose that 
\begin{equation}
d^2 > cN. \label{N0-case2}
\end{equation} 
Let us choose any finite set $\om \subset Z_{\mathbb R}(P)$, with $P \in \mathcal P_{d}(2)$, $P \not\equiv 0$ and $\# \om =  N$. In this case, we first identify a low degree polynomial which replaces the role of $P$. Let $D$ denote the smallest integer exceeding $c_0 {N}^{\frac{1}{2}}$, where $0<c_0<1$ is a small constant defined by 
\begin{equation} \label{choice-of-c0} 
8^4 A_1 c_0^2 = c.  
\end{equation}  
Applying Theorem \ref{thm:GK} with this $D$ and $n=2$, we find a nontrivial polynomial $P_0$ of degree at most $D \leq 2c_0 {N}^{\frac{1}{2}}$ such that
\begin{align} 
& M_0 := \# \bigl\{\text{components $O_j$ in $\mathbb R^2 \setminus Z(P_0)$} \bigr\} \leq A_1 D^2 \leq 4 A_1 c_0^2 N, \label{number-of-components} \\ 
&\sup_j \#(\om \cap O_j) \leq A_1 N D^{-2} \leq A_1 c_0^{-2}. \label{number-of-points-per-component} 
\end{align} 
\vskip0.1in
\noindent We set $\om_j = \om \cap O_j$, and note that this gives rise to the decomposition 
\begin{align} 
\Omega = \Omega^{\ast} \cup \Omega^{\ast \ast}, \; &\text{ where } \; \Omega^{\ast} := \bigcup_j \Omega_j \text{ and } \Omega^{\ast \ast} := \Omega \cap Z(P_0); \text{ consequently } \nonumber \\  
& ||{H_\om}||_{2 \rightarrow 2} \leq ||H_{\Omega^{\ast}}||_{2 \rightarrow 2}  + ||H_{\Omega^{\ast \ast}}||_{2 \rightarrow 2}. \label{Hom-split} 
 \end{align} 
The induction will close provided we verify the following two inequalities: 
\begin{equation} \label{ineq-12} ||H_{\Omega^{\ast \ast}}||_{2 \rightarrow 2} \leq \frac{A}{2} \sqrt{d} \log N, \qquad  ||H_{\Omega^{\ast}}||_{2 \rightarrow 2} \leq \frac{A}{2} \sqrt{d} \log N. \end{equation} 
\vskip0.1in
\noindent Let us prove the first inequality in \eqref{ineq-12}. Note that \eqref{N0-case2} and \eqref{choice-of-c0} ensure that
\begin{equation}\label{eqn:degree0}
 D \leq 2c_0 N^{1/2} \leq d/4^2.
\end{equation}
Since $(N,D) < (N,d)$, we may apply the induction hypothesis to get 
\begin{align*} 
\norm{H_{\om^{\ast \ast}}}_{2 \rightarrow 2}  \leq \mathfrak C^{\ast}_2(N;D)  \leq A {D}^{1/2} \log N  \leq \frac{A}{2} \sqrt{d} \log N. 
\end{align*}
\vskip0.1in  
\noindent We turn now to the second inequality in \eqref{ineq-12}, i.e., the contribution from $\Omega^{\ast}$. This will be obtained using Theorem \ref{thm:ortho}. The set $\mathcal O$ is chosen so that $\mathcal O \subseteq \Omega^{\ast} \subseteq Z(P)$, with $\#(\mathcal O \cap O_j) =1$ for every $j$. Therefore, by \eqref{number-of-components},  $\#(\mathcal O) := M_0 \leq 4A_1 c_0^2 N$. Since $\deg P_0 \leq D$, almost every line in $\mathbb R^2$ intersects $Z(P_0)$ in at most $D$ points; hence it can intersect at most $D+1$ components $\{O_j\}$. Therefore, by Theorem \ref{thm:ortho} and the bound on $\#( \om_j)$ from \eqref{number-of-points-per-component}, 
\begin{align} 
\norm{H_{\Omega^{\ast}}}_{2 \rightarrow 2} &\leq ||H_{\mathcal O}||_{2 \rightarrow 2} + \sqrt{D+1} \Bigl( \max_j ||H_{\Omega_j}||_{2 \rightarrow 2} + 1 \Bigr) \nonumber \\ 
&\leq ||H_{\mathcal O}||_{2 \rightarrow 2} + \sqrt{D+1} \Bigl( \max_j \#(\Omega_j) + 1 \Bigr) \nonumber \\
&\leq ||H_{\mathcal O}||_{2 \rightarrow 2}+ 2 \sqrt{D} (A_1 c_0^{-2} + 1)\nonumber \\ 
&\leq \mathfrak C_2^{\ast}(M_0; d) + C_0 \sqrt{D}, \label{Omega-star}
\end{align} 
where $C_0 := 2(A_1 c_0^{-2}+1)$ is an absolute constant depending on $c_0$. 
By choosing a sufficiently large constant $A$ satisfying 
\begin{equation}\label{choice-of-A-3}
A \geq C_0,
\end{equation}
we see, from \eqref{eqn:degree0} and\eqref{Omega-star}, that
\[
\norm{H_{\Omega^{\ast}}}_{2 \rightarrow 2} \leq \mathfrak C_2^{\ast}(M_0; d) + \frac{A}{4} \sqrt{d} \leq \mathfrak C_2^{\ast}(M_0; d) + \frac{A}{4} \sqrt{d} \log N.
\]
\vskip0.1in
\noindent The second inequality in \eqref{ineq-12} is therefore a consequence of 
\[ \mathfrak C_2^{\ast}(M_0; d) \leq \frac{A}{4} \sqrt{d} \log N. \]
To prove this, we observe that \eqref{choice-of-c0} and \eqref{number-of-components} imply that $M_0 < N/2$. Hence $(M_0, r) < (N, r)$ for any choice of $r$. In addition, \eqref{choice-of-c0} implies that $2 \sqrt{M_0} < d/2 $, in view of \eqref{N0-case2} and \eqref{number-of-components}. Therefore, the identity \eqref{C2-dlarge-N} dictates that $\mathfrak C_2^{\ast}(M_0; d) 
= \mathfrak C_2^{\ast}(M_0; 2 \sqrt{M_0})$. Applying the induction hypothesis on  $(M_0, 2\sqrt{M_0}) < (N, d)$, we obtain
\begin{align*} 
\mathfrak C_2^{\ast}(M_0; d) 
= \mathfrak C_2^{\ast}(M_0; 2 \sqrt{M_0})
&\leq  A \sqrt{2} M_0^{1/4}\log M_0 \\
&\leq A \sqrt{2} (4A_1 c_0^2 )^{1/4} N^{1/4} \log N  \\
&\leq A(2^4 c_0^{2} A_1)^{1/4}  (c^{-1} d^2)^{1/4} \log N \\ 
&= \frac{A}{4} \sqrt{d} \log N,
\end{align*} 
where we used \eqref{number-of-components}, \eqref{N0-case2}, and \eqref{choice-of-c0}. 
This completes the estimation for $||H_{\Omega^{\ast}}||_{2 \rightarrow 2}$, and hence the proof. As a summary of the size requirements for the constant $A$, we note that $A$ is chosen to satisfy \eqref{choice-of-A-1}, \eqref{A4}, and \eqref{choice-of-A-3}.  
\end{proof} 

\section{Improved estimates in $\mathbb R^3$, Part 2: Proof of Theorem \ref{thm:3dgen}} \label{sec:3d}
\noindent Theorems \ref{thm:GK} and \ref{thm:3d}, along with Theorem \ref{thm:ortho}, are the main ingredients of this proof.   
\begin{proof} 

Set $\omega(N) := {h(N)}/{\log N}$. The assumptions made in \eqref{conditions-on-h} on $h$ ensure that for every $\epsilon \in (0,1)$, there exists an integer $R_{\epsilon} \gg 1$ such that 
\begin{equation} \label{what-is-epsilon}
\omega(N) \leq \epsilon \ll 1,  \, N \omega(N)^4 > \epsilon^{-1}, \,  h(N) > \epsilon^{-1}, \, h((\log N)^4) < \epsilon h(N). 
\end{equation} 
for all large $N \geq R_{\epsilon}$. For us, $\epsilon > 0$ will be an absolute constant whose exact value will be determined in the sequel; see \eqref{condition-epsilon-1} below.  As in Theorem \ref{thm:3d}, we also define 
\begin{equation}
\mathfrak C_2(N) := \bigl\{ ||H_{\Omega}||_{2 \rightarrow 2} :  \Omega \subseteq \mathbb R^2, \; \#(\Omega) = N \bigr\}. 
\end{equation}   
The conclusion \eqref{eqn:3dgen} of Theorem \ref{thm:3dgen} is equivalent to finding an absolute constant $C$, depending only on $h$ and $\epsilon$, such that 
\begin{equation} \label{eqn:3dgen:rephrased}
\mathfrak C_2(N) \leq C N^{\frac{1}{4}} h(N). 
\end{equation}  
\vskip0.1in
\noindent We prove \eqref{eqn:3dgen:rephrased} by induction on $N$. Without loss of generality and in view of \eqref{conditions-on-h}, we may assume that the constant $C$ in \eqref{eqn:3dgen:rephrased} is large enough to satisfy $N \leq C N^{1/4} h(N)$ for all $N \leq R_{\epsilon}$. For such $N$, the inequality \eqref{eqn:3dgen:rephrased} would follow from the trivial bound $||H_{\Omega}||_{2 \rightarrow 2} \leq N$. This covers the base of the induction. \vskip0.1in
\noindent Suppose now that $N > R_{\epsilon}$ and that \eqref{eqn:3dgen:rephrased} holds for all $N' < N$. Given a finite set $\Omega \subseteq \mathbb R^2$ of cardinality $N$, we choose 
\[ d_0 = \sqrt{N} (\omega(N))^2.\] 
Theorem  \ref{thm:GK} then generates a nontrivial polynomial $P_0$ of degree at most $d_0$ such that $\mathbb R^2 \setminus Z_{\mathbb R}(P_0)$ is the disjoint union of at most $A_1 d_0^2 = A_1 N (\omega(N))^4$ connected components  $O_j$, each containing at most $A_1 N d_0^{-2} = A_1 (\omega(N))^{-4}$ points of $\Omega$. As in the proof of Theorem \ref{thm:3d}, we set $\Omega_j = \Omega \cap O_j$, and write \[ \Omega = \Omega^{\ast} \bigcup \Omega^{\ast \ast}, \quad \text{ where } \quad\Omega^{\ast} = \bigcup_j \Omega_j \text{ and } \Omega^{\ast \ast} = \Omega \cap Z_{\mathbb R}(P_0). \]  
This leads to the same decomposition of $H_{\Omega}$ as in \eqref{Hom-split}. As before, the inductive step will close if we are able to show that for a sufficiently large constant $C$ depending only on $h$ and $\epsilon$, the following estimates hold: 
\begin{equation} ||H_{\Omega^{\ast}}||_{2 \rightarrow 2} \leq \frac{C}{2} N^{\frac{1}{4}} h(N) \quad \text{ and } \quad ||H_{\Omega^{\ast \ast}}||_{2 \rightarrow 2} \leq \frac{C}{2} N^{\frac{1}{4}} h(N). \label{two-estimates-to-show} 
\end{equation}
\noindent Theorem \ref{thm:3d} controls the second term in \eqref{two-estimates-to-show}, namely the contribution from $\Omega^{\ast \ast}$. Applying the conclusion \eqref{Zp-3d} of this theorem, we obtain 
\[ ||H_{\Omega^{\ast \ast}}||_{2 \rightarrow 2} \leq A \sqrt{d_0} \log N \leq A \bigl[\sqrt{N} \omega(N)^2 \bigr]^{\frac{1}{2}} \log N = A N^{\frac{1}{4}} h(N). \]   
As long as the constant $C$ in \eqref{eqn:3dgen} is chosen larger than $2A$, where $A$ is the absolute constant from \eqref{Zp-3d}, the second inequality in \eqref{two-estimates-to-show} follows.   
\vskip0.1in
\noindent The analysis of the first term in \eqref{two-estimates-to-show} is very similar to its counterpart in Theorem \ref{thm:3d}, so we only sketch the details. We apply Theorem \ref{thm:ortho} with $\Omega_j = \Omega \cap O_j$, $||E||_{L^{\infty}} \leq d_0+1$ and $\mathcal O \subseteq \Omega$. Suppose that the constant $\epsilon\in (0,1)$ in \eqref{what-is-epsilon} is chosen small enough so that 
\begin{equation} \label{condition-epsilon-1}
10A_1^{1/4} \epsilon < 1.
\end{equation} 
Then, on one hand, the assumption $N \geq R_{\epsilon}$ and \eqref{what-is-epsilon} yield that the number of components $O_j$, i.e., $\#(\mathcal O)$ satisfies
\[  \#(\mathcal O) \leq A_1 d_0^2 = A_1 (\omega(N))^4 N \leq A_1 \epsilon^4 N.\]
On the other hand, the third relation in \eqref{what-is-epsilon} combined with \eqref{condition-epsilon-1} shows that for every $j$,
 \[\#(\Omega_j) \leq A_1 (\omega(N))^{-4} < A_1 \epsilon^4 (\log N)^4 < (\log N)^4.\] Thus, we have $\#(\mathcal O) \leq N/2$ and $\#(\Omega_j)\leq N/2$, so the induction hypothesis in $N$ applies to both $H_{\mathcal O}$ and $H_{\Omega_j}$. 
Invoking the relation \eqref{almost-ortho} from Theorem \ref{thm:ortho}, and combining it with the induction hypothesis and the bounds for $\#(\mathcal O)$ and $\#(\Omega_j)$ given above, we obtain 
\begin{align*}
||H_{\Omega^{\ast}}||_{2 \rightarrow 2} 
&\leq ||H_{\mathcal O}||_{2 \rightarrow 2} + ||E||_{\infty}^{\frac{1}{2}} \Bigl( \sup_j ||H_{\Omega_j}||_{2 \rightarrow 2} + 1 \Bigr) \\ 
&\leq \mathfrak C_2(A_1 \epsilon^4 N) + \sqrt{d_0 +1} \Bigl[ \mathfrak C_2 \bigl(A_1 \omega(N)^{-4} \bigr) + 1 \Bigr] \\ 
&\leq C (A_1 \epsilon^4 N)^{\frac{1}{4}} h(A_1 \epsilon^4 N) + 4 C \sqrt{d_0} (A_1  \omega(N)^{-4})^{\frac{1}{4}}  h(A_1 \omega(N)^{-4}) \\ 
&\leq C (A_1 \epsilon^4 N)^{\frac{1}{4}} h(N) + 4C A_1^{\frac{1}{4}} N^{\frac{1}{4}}h((\log N)^{4}))  \\
&= C N^{\frac{1}{4}} h(N) \Bigl[ A_1^{1/4} \epsilon + 4 A_1^{\frac{1}{4}} \frac{h((\log N)^{4}))}{h(N)}  \Bigr] \\
&\leq C N^{\frac{1}{4}} h(N) \Bigl[ A_1^{1/4} \epsilon + 4 A_1^{\frac{1}{4}} \epsilon \Bigr] < \frac{C}{2} N^{\frac{1}{4}} h(N),  
\end{align*} 
where in the last display we have used the third requirement in \eqref{what-is-epsilon}, and also \eqref{condition-epsilon-1}. This proves the first estimate in \eqref{two-estimates-to-show} and hence completes the proof of the theorem.  
\end{proof}
\section{Direction sets contained in varieties: Proof of Theorem \ref{thm:higherdim}} \label{proof-thm-highd-section} 
\subsection{Algebraic geometry preliminaries}\label{sec:poly}
It remains to prove Theorem \ref{thm:higherdim}. Its proof utilizes certain tools, some of which are classical in the algebraic geometry literature, and some that have emerged from recent developments in polynomial partitioning. We collect the relevant facts and definitions in this section. The proof of Theorem \ref{thm:higherdim} is given in section \ref{section:proof:higherdim}. 
\subsubsection{Definitions} 
An algebraic {\em{variety}} $V$ in $\C^n$ is a set of the form
\[ V = Z_{\C}(P_1, \ldots,P_{l}) := \{ z = (z_1, \cdots, z_n) \in \C^{n} : P_1(z)=\cdots = P_{l}(z)=0 \}, \]
where $P_1,\cdots,P_l\in \C[z_1,\cdots, z_n]$ are polynomials. A variety is said to be {\em{irreducible}} if it cannot be written as the union of two strictly smaller varieties. It is well-known \cite[Proposition I.5.3]{Mumford} that any variety can be uniquely expressed as the union of irreducible varieties, also called irreducible components. Each (complex) variety $V$ in $\C^n$ generates a real variety $V(\R)$ in $\R^n$, by setting $V(\R) := V \cap \bigl[ \R^n + i \{\vec{0}\} \bigr]$; in other words, \[V(\mathbb R) = Z_{\mathbb R}(P_1, \ldots, P_l) =  \{ x  \in \R^{n} : P_1(x)=\cdots = P_{l}(x)=0 \}. \]  
In what follows, we will always be working in $\mathbb R^n$. So, even though $V$ is a priori defined in $\mathbb C^n$ and its intrinsic properties (such as dimension and degree) will be defined therein, our analysis will take place on $V(\mathbb R)$. Similarly, for a given polynomial $P \in \mathbb C[z_1, \cdots, z_n]$, we will focus on $Z_{\mathbb R}(P)$. Henceforth, we will drop the suffix $\mathbb R$, and denote the zero set of $P$ in $\mathbb R^n$ simply by $Z(P)$.   
\vskip0.1in
\noindent The concepts of dimension and degree are central to the notion of a variety in $\C^n$. We recall them here, following the treatment of \cite[Section 2]{MP}. More extensive discussions may be found in \cite{{CLO}, {Harris}, {Hartshorne}, {Kempf}}. The {\em{dimension}} $\dim V$ of a variety $V$ in $\C^n$ is the smallest integer $0\leq m\leq n$ such that a generic $(n-m)$-dimensional complex affine subspace $S$ of $\C^n$ intersects $V$ in finitely many points. The {\em{degree}} of $V$ is the number of intersections, which is the same for all generic $S$. To clarify the meaning of ``generic", let us consider subspaces $S = S(a)$ of the form \[ z_{i + n-m} = a_{i0} + \sum_{j=1}^{n-m} a_{ij} z_j, \quad 1 \leq i \leq m.\] We call a subspace $S = S(a)$ ``generic'' if the vector of constants $a = (a_{ij} : 1 \leq i \leq m, 1 \leq j \leq n-m)$ does not lie in the zero set of a certain nontrivial polynomial depending on $V$. Thus, almost all subspaces $S$ is the sense of measure are generic. 
Alternatively and equivalently, one can define the dimension of an irreducible variety $V$ to be the largest integer $0\leq m \leq n$  for which there exists a sequence 
\[ \emptyset \neq V_0 \subsetneq V_1 \subsetneq \cdots \subsetneq V_m = V \]
of irreducible varieties between $\emptyset$ and $V$. When $V$ has several irreducible components $\{ V_j\}$, then $\dim V$ is defined to be the maximum of $\dim V_j$. We note that $\mathbb C^n$ is itself an algebraic variety, whose dimension is $n$ and whose degree is $1$.
\vskip0.1in
\subsubsection{Dimension of intersection of varieties} 
The proof of Theorem \ref{thm:higherdim} involves induction on the dimension of an algebraic variety. The following result, a consequence of the well-known principal ideal theorem \cite[Theorem 2.6.3]{Kempf}, \cite[Section I.8]{Mumford}, provides an ingredient for the inductive step, by ensuring a dimension drop in the intersection of the original variety with the zero set of certain polynomials.   
\begin{lem} \label{lem:dim-intersection} 
Let $V$ be an $m$-dimensional variety in $\mathbb C^n$, and let $Q \in \mathbb C[z_1, \cdots, z_n]$ be a polynomial that does not vanish identically on any irreducible component of $V$. Then $\text{dim}(V \cap Z(Q)) < \text{dim}(V)$.  
\end{lem} 
\subsubsection{Degree of intersection of varieties}
We will also need to control the degree of a variety arising from the intersection of a given variety with the zero set of a polynomial. 
\begin{lem} (A generalized Bezout's inequality \cite[Lemma 2.2]{MP}) \label{gen-Bezout} 
Let $V \subseteq \mathbb C^n$ be an irreducible variety of dimension $m$, and let $P \in \mathbb C[z_1, \cdots, z_n]$ be a polynomial that does not vanish identically on $V$. Suppose that $W_1, \cdots, W_k$ are the irreducible components of $V \cap Z_{\mathbb C}(P)$. Then each of the components $W_i$ has dimension $m-1$, and 
\[ \deg(V \cap Z_{\C}(P) )= \sum_{i=1}^{k} \deg (W_i) \leq \deg(V) \deg(P). \] 
\end{lem}  
\noindent {\em{Remarks: }} The above bound extends, in particular, to possibly reducible varieties $V$ in $\C^n$ such that each irreducible component of $V$ has dimension $m$. To see this, let $\{V_j\}$ be the irreducible components of $V$. By assumption, each $V_j$ has dimension $m$. Then by applying Lemma \ref{gen-Bezout} to each $V_j \cap Z_{\C}(P)$, we get
\begin{equation}\label{eqn:Bezout}
\begin{split}
\deg(V \cap Z_{\C}(P) ) &\leq \sum_j \deg(V_j \cap Z_{\C}(P) ) \\
&\leq \sum_j \deg(V_j) \deg (P) = \deg (V) \deg(P).
\end{split}
\end{equation}

\subsubsection{Polynomial partitioning} The polynomial partitioning theorem due to Guth and Katz, namely Theorem \ref{thm:GK}, was an important ingredient in our proof of Theorem \ref{thm:3dgen}. Not surprisingly, our proof of the higher dimensional variant Theorem \ref{thm:higherdim} requires a refinement of similar partitioning techniques. Using polynomials for efficient partitioning of finite point sets in low dimensional varieties is an active avenue of research; see e.g. \cite{MP, BS, Za, Walsh}. Among these, the following generalization of Theorem \ref{thm:GK}, due to Matou\v{s}ek and Pat\'{a}kov\'{a} \cite{MP}, will be a key component of our proof.

\begin{theorem}[{\cite[Lemma 3.1]{MP}}] \label{thm:polyvar} 
Let $V$ be any variety in $\C^n$ such that each irreducible component of $V$ has dimension $m$. Assume that $\om\subset V(\R)$ is a finite set of $N$ elements. Then for any given $D\geq 1$, there is a polynomial $P\in \R[x_1,\ldots,x_{n}]$ of degree at most $D$ such that $P$ does not vanish identically on each irreducible component of $V$ and each connected component of $\R^{n} \setminus Z(P)$ contains at most $C_m N D^{-m}$ elements of $\om$, where $C_m$ is an absolute positive constant that depends only on $m$. 
\vskip0.1in
\noindent  
\end{theorem}
\noindent {\em{Remarks: }}
\begin{enumerate}[1.]
\vskip0.1in
\item \label{rem:polyvar}
Since each connected component of $V(\R) \setminus Z(P)$ is a subset of some connected component of $\R^n \setminus Z(P)$, one can also conclude in Theorem \ref{thm:polyvar} that each connected components of $V(\R) \setminus Z(P)$ contains at most $C_m N D^{-m}$ elements of $\om$. In the proof of Theorem \ref{thm:higherdim}, in order to analyze the subset of $\Omega$ contained in $V(\R)$, we will apply Theorem \ref{thm:ortho} with the connected sets $\{O_j \}$ being the connected components of $V(\R) \setminus Z(P)$.
\vskip0.1in 
\item The strength of Theorem \ref{thm:polyvar} lies in its applicability to an arbitrary algebraic variety of any dimension, regardless of whether it is irreducible or not. It is also important for our applications that the constant $C_m$ provided by Theorem \ref{thm:polyvar} is uniform for all $m$-dimensional varieties $V$; An inspection of its proof in \cite{MP} shows that $C_m$ depends only on the constant $A_1$ from Theorem \ref{thm:GK} in $\mathbb R^m$.  
\vskip0.1in
\item For {\em{irreducible varieties of large degree}}, the bound $O(N D^{-m})$ can sometimes be replaced by a stronger bound depending on the degree; see \cite{BS,Walsh}. While this could potentially be useful in obtaining a result more precise than Theorem \ref{thm:higherdim}, this strategy seems to require a good quantitative bound on the number of irreducible components of a given variety. We do not pursue this direction here. 
\end{enumerate}

\subsubsection{Connected components in a real algebraic variety} In view of Theorem \ref{thm:ortho} and remark \ref{rem:polyvar} above, we will need to control the number of connected components of $V(\R) \setminus Z(P)$ as well as the number of components intersecting a generic hyperplane $Z(P_u)$. There are many results in the literature that address such issues. In particular, Barone and Basu \cite{BB, BB2} have given a general bound on the number of components depending on various parameters. A nice exposition of a simpler version of their result, which suffices for our purposes, appears in the work of Solymosi and Tao \cite[Theorem A.2]{ST}. It can be stated, combined with \cite[Lemma 4.2]{ST}, as follows.
\begin{theorem}(\cite[Theorem A.2]{ST}) \label{thm:ST} 
Let $V$ be any $m$-dimensional variety in $\C^n$ of degree at most $d$ for some $1 \leq m \leq n$ and $d\geq 1$. Assume that $P\in \R[x_1,\cdots,x_n]$ is a polynomial of degree at most $D$ for some $D\geq 1$. Then the set $V(\R) \setminus Z(P)$ has at most $R_{n,d} D^m$ connected components, where $R_{n, d}$ is a positive constant that depends only on $n,d$.
\end{theorem}

\noindent Our next task is to estimate the number of connected components of $V(\mathbb R) \setminus Z(P)$ that intersect a generic hyperplane. This is a key step in the application of Theorem \ref{thm:ortho}, leading to the estimation of the quantity $E$ therein.
\begin{prop} \label{lem:boundE} 
Let $V$ be a variety in $\C^n$ of degree at most $d$ such that each irreducible component of $V$ has dimension $m$. Suppose that $P\in \R[x_1,\cdots,x_n]$ is a polynomial of degree at most $D$ for some $D\geq 1$. For $u \in \mathbb S^n$, let $E(u)$ be the number of connected components of $V(\R)\setminus Z(P)$ intersecting $Z(P_u)$, where $P_u(y) = u\cdot \inn{y,1}$. Then 
\begin{equation} \label{E-mnd} 
\norm{E}_{L^\infty(\mathbb S^n)} \leq R_{n,d} D^{m-1}, 
\end{equation} 
with the constant $R_{n, d}$ provided by Theorem \ref{thm:ST}. 
\end{prop}
\begin{proof}
Fix any $u \in \mathbb S^n$. Let $\mathbb O = \{O_j\}$ and $\mathbb O(u) = \{O_k(u)\}$ denote respectively the finite collections of nonempty connected components of $V(\mathbb R) \setminus Z(P)$ and $\bigl[ V(\mathbb R) \cap Z(P_u) \bigr] \setminus Z(P)$. Our main claim is that, for every index $j$ such that $O_j \cap Z(P_u)$ is nonempty, there exists at least one index $k$ such that   
\begin{equation} \label{inclusion-kj} 
O_k(u) \subseteq O_j \cap Z(P_u). 
\end{equation} 
We will prove this claim in a moment. Assuming this for now, we deduce from \eqref{inclusion-kj} that  
\begin{align*}
E(u) &= \# \bigl\{j : O_j \cap Z(P_u) \ne \emptyset \bigr\} \\
&\leq \# \bigl\{k :  O_k(u) \subseteq O_j \cap Z(P_u) \; \text{ for some } j \bigr\} \leq \#(\mathbb O(u))
\end{align*}
In Lemma \ref{lem:dim} of the appendix, we show that $\dim [Z_{\C}(P_u)\cap V] \leq m-1$ for almost every $u\in \mathbb S^n$. Since $\deg V \leq d$ and $\deg P_u = 1$, we know that $\deg(Z_{\C}(P_u)\cap V) \leq d$ by \eqref{eqn:Bezout}. Hence, applying Theorem \ref{thm:ST} with $m$ replaced by $(m-1)$ to the variety $Z_{\C}(P_u)\cap V$, we get $\#(\mathbb O(u)) \leq R_{n, d} D^{m-1}$ for almost every $u\in S^n$. This leads to the desired bound \eqref{E-mnd}. 
\vskip0.1in
\noindent It remains to prove the claim resulting in \eqref{inclusion-kj}. We observe that 
\[ \bigcup_k O_k(u) = \bigl[ Z(P_u) \cap V(\mathbb R) \bigr] \setminus Z(P) =  \bigcup_j \bigl[ Z(P_u) \cap O_j \bigr].\]
Hence for any index $j$ such that $Z(P_u) \cap O_j \ne \emptyset$, there must exist some index $k$ such that $Z(P_u) \cap O_j \cap O_k(u) \ne \emptyset$. We intend to show that \eqref{inclusion-kj} holds for this pair $(k, j)$. For this, we note that $O_k(u)$ can be written as a disjoint union,
\[ O_k(u) = \bigcup_{j'} \bigl[ Z(P_u) \cap O_{j'} \cap O_k(u) \bigr].\]
Each $O_{j'}$ is by definition both open and closed in $V(\mathbb R) \setminus Z(P)$; hence each of the sets $Z(P_u) \cap O_{j'} \cap O_k(u)$ is both open and closed in $O_k(u) \cap \bigl[(V(\mathbb R) \cap Z(P_u)) \setminus Z(P) \bigr] = O_k(u)$. Since $O_k(u)$ is connected, this implies that $Z(P_u) \cap O_{j'} \cap O_k(u)$ can be nonempty for only one of the indices $j'$, namely for $j' = j$. Thus $O_k(u) = Z(P_u) \cap O_{j} \cap O_k(u)$, which is equivalent to the claimed relation \eqref{inclusion-kj}.  
\end{proof}

\subsection{Proof of Theorem \ref{thm:higherdim}} \label{section:proof:higherdim}

We define 
\[ \mathfrak C_{\text{alg}}(N; m,n,d) := \sup \bigl\{ ||H_{\Omega}||_{2 \rightarrow 2} : \Omega \subseteq V(\mathbb R),  V \in \mathcal V(m,n,d), \; \#(\Omega) \leq N \bigr\} \] 
and aim show that
\begin{equation}\label{eq:goal_reformulate}
\mathfrak C_{\text{alg}}(N;m,n, d) \leq 
\begin{cases}
 d &\text{ when } m=0, \\ 
 A_{\epsilon}(m, d) N^{\frac{m-1}{2m} +\epsilon} &\text{ when } 1\leq m \leq n.
\end{cases}
\end{equation}
As in the proof of Theorem \ref{thm:3d}, we will establish the relation \eqref{eq:goal_reformulate} by induction on $(m, N)$ using the lexicographic ordering, with $n$ and $d$ fixed. 
\vskip0.1in 
\noindent The initializing step of the induction corresponds to $m=0$. By definition, the cardinality of any zero-dimensional variety equals its degree. 
Therefore, if $\Omega \subset V \in \mathcal V(0, n, d)$ is a finite set, then by the trivial estimate we have,
\[ ||H_{\Omega}||_{2 \rightarrow 2} \leq \#(\Omega) \leq \#(V(\mathbb R)) \leq \#(V) = d.  \] 
This establishes \eqref{eq:goal_reformulate}, as required. 
\vskip0.1in 
\noindent We continue to the inductive step. Let us fix $n \geq 2$, $1 \leq m \leq n$ and an arbitrary $0 < \epsilon < 1$. Suppose that \eqref{eq:goal_reformulate} has been established for $\mathfrak C_{\text{alg}}(N'; m', n, d)$ for all $(m', N') < (m, N)$ and for all $d$. We will prove \eqref{eq:goal_reformulate} for $\mathfrak C_{\text{alg}}(N; m, n, d)$, for a sufficiently large absolute constant $A_{\epsilon}(m,d)$.  Accordingly, we choose $V \in \mathcal V(m, n, d)$ of dimension $m$, a direction set $\Omega \subseteq V$ with $\#(\Omega) = N$, and aim to show that 
\begin{equation}  ||H_{\Omega}||_{2 \rightarrow 2} \leq A_{\epsilon}(m, d) N^{\frac{m-1}{2m} + \epsilon}.  \label{aim-to-show-eq:goal}\end{equation}
\vskip0.1in
\noindent We first classify the irreducible components of $V$ according to their respective dimensions, and write $V = U_{m} \cup V_m$, where each irreducible component of $V_{m}$ (respectively $U_{m}$) is of dimension $m$ (respectively $<m$). Intersecting both sides of this relation with $\mathbb R^n$ and then with $\Omega$ results in the following decompositions:
\begin{align*}  
V(\mathbb R) &=   V_{m}(\mathbb R) \cup U_{m}(\mathbb R) , \qquad \Omega = \Omega(V_m) \cup \Omega(U_m) , \text{ where } \\ 
\Omega(V_m)  &= \Omega \cap V_m(\mathbb R), \qquad \Omega(U_m)  = \Omega \cap U_m(\mathbb R).  
\end{align*}    
As a result, 
\begin{align} 
||H_{\Omega}||_{2 \rightarrow 2} &\leq  ||H_{\Omega(V_m)}||_{2 \rightarrow 2} +  ||H_{\Omega(U_m)}||_{2 \rightarrow 2} \nonumber \\ &\leq ||H_{\Omega(V_m)}||_{2 \rightarrow 2} + \mathfrak C_{\text{alg}}(N; m-1, n,d)  \nonumber \\ &\leq ||H_{\Omega(V_m)}||_{2 \rightarrow 2} + \begin{cases} d &\text{ if } m = 1, \\ A_{\epsilon}(m-1,  d) N^{\frac{m-2}{2(m-1)} + \epsilon} &\text{ if } 2 \leq m \leq n,  \end{cases} \nonumber \\ 
&\leq  ||H_{\Omega(V_m)}||_{2 \rightarrow 2} + \frac{1}{2} A_{\epsilon}(m, d) N^{\frac{m-1}{2m} + \epsilon}. \label{remove-Um}
\end{align} 
At the third step above, we have applied the induction hypothesis on $\Omega(U_m)$, with $(m', N') = (m-1, N)$. Since the exponent function $m \mapsto (m-1)/(2m)$ is increasing with $m$, the last step follows if we choose 
\begin{equation*} 
\frac{A_{\epsilon}(m, d)}{2} > \begin{cases} d &\text{ if } m =1, \\ A_{\epsilon}(m-1, d) &\text{ if } 2 \leq m \leq n. \end{cases}  
\end{equation*} 
In view of \eqref{remove-Um}, the desired estimate \eqref{aim-to-show-eq:goal} will follow from
\begin{equation} \label{aim-to-show-eq:goal2} 
||H_{\Omega(V_m)}||_{2 \rightarrow 2} \leq \frac{1}{2} A_{\epsilon}(m, d) N^{\frac{m-1}{2m} + \epsilon}.
\end{equation}  
We set about proving this.
\vskip0.1in
\noindent Let $D=D_{\epsilon,m,n,d}$ be a large integer to be specified shortly (in inequalities \eqref{D1} and \eqref{D2} below). By Theorem \ref{thm:polyvar}, there exists a polynomial of degree at most $D$ such that $P$ does not vanish identically on any irreducible component of $V_m$ and each connected component of $V_m(\mathbb R) \setminus Z(P)$ contains at most $C_m ND^{-m}$ elements of $\Omega(V_m)$. As in the proof of Theorem \ref{thm:3d}, this results in a decomposition of $\Omega(V_m)$, and the corresponding operator:
\begin{align*} 
\Omega(V_m) &= \Omega^{\ast}(V_m) \cup \Omega^{\ast \ast}(V_m),\text{ where } \nonumber \\  \Omega^{\ast}(V_m) &= \Omega(V_m)  \setminus Z(P), \text{ and }  \; \Omega^{\ast \ast}(V_m) = \Omega(V_m) \cap Z(P). 
\end{align*} 
In order to prove \eqref{aim-to-show-eq:goal2}, it therefore suffices to establish the following two inequalities: 
\begin{align} \label{aim-to-show-eq:goal3} 
||H_{\Omega^{\ast}(V_m)}||_{2 \rightarrow 2} &\leq \frac{1}{4} A_{\epsilon}(m, d) N^{\frac{m-1}{2m} + \epsilon}, \\ 
\label{aim-to-show-eq:goal4}
||H_{\Omega^{\ast \ast}(V_m)}||_{2 \rightarrow 2} &\leq \frac{1}{4} A_{\epsilon}(m,d) N^{\frac{m-1}{2m} + \epsilon}, 
\end{align} 
\vskip0.1in
\noindent We start with \eqref{aim-to-show-eq:goal3}, namely the contribution from $\Omega^{\ast}(V_m)$. The key here is once again Theorem \ref{thm:ortho}. As preparation for Theorem \ref{thm:ortho}, let $\{O_j\}$ be the collection of connected components of $V(\mathbb R) \setminus Z(P)$, and let $\mathcal O$ denote the collection of points obtained by selecting a single point $v_j \in O_j$ for each $j$. Then 
\[ \Omega^{\ast}(V_m) = \bigcup_j \Omega_j(V_m) \quad \text{ where } \quad \Omega_j(V_m) := \Omega(V_m) \cap O_j. \] 
We estimate the contribution of $\mathcal O$ by the trivial bound: 
\begin{align} 
||H_{\mathcal O}||_{2 \rightarrow 2} &\leq \#(\mathcal O) = \text{ number of  connected components of  $V(\mathbb R) \setminus Z(P)$} \nonumber \\ 
&\leq  R_{n,d} D^m. \label{HO} 
\end{align}  
The last inequality is a consequence of Theorem \ref{thm:ST}. We also observe that each $\Omega_j(V_m)$ contains at most $C_m N D^{-m}$ elements, per Theorem \ref{thm:polyvar}. Choosing $D$ large enough so that 
\begin{equation} \label{D1}
C_m D^{-m} < 1/2
\end{equation} 
allows us to apply the induction hypothesis on $(m', N') = (m, C_m N D^{-m})$, resulting in the estimate
\begin{align} 
||H_{\Omega_j}||_{2 \rightarrow 2} &\leq \mathfrak C_{\text{alg}}(C_m N D^{-n}; m, n, d) \nonumber \\
&\leq A_{\epsilon}(m, d) \bigl( C_m N D^{-m} \bigr)^{\frac{m-1}{2m} + \epsilon}. \label{H-Omj-higherdim}
\end{align}  
By Proposition \ref{lem:boundE}, we also have that \begin{equation} \label{Ebound-higherdim} ||E||_{L^{\infty}(\mathbb S^n)} \leq R_{n,d} D^{m-1}, \end{equation}  where $E(u)$ denotes the number of components $O_j$ intersected by $Z(P_u)$. Substituting \eqref{HO}, \eqref{H-Omj-higherdim} and \eqref{Ebound-higherdim} into \eqref{almost-ortho} yields 
\begin{align*}  
||H_{\Omega^{\ast}(V_m)}||_{2 \rightarrow 2} &\leq ||H_{\mathcal O}||_{2 \rightarrow 2} + ||E||_{\infty}^{\frac{1}{2}} \bigl(\sup_j ||H_{\Omega_j}||_{2 \rightarrow 2} + 1 \bigr) \\ &\leq R_{n, d} D^m + 2 \bigl( R_{n, d} D^{m-1} \bigr)^{\frac{1}{2}} A_{\epsilon}(m, d)  \bigl(C_m N D^{-m}\bigr)^{\frac{m-1}{2m} + \epsilon} \nonumber \\ 
&\leq R_{n, d} D^m + \bigl[ 2 R_{n, d}^{\frac{1}{2}} C_m^{\frac{m-1}{2m}+\epsilon} D^{-m \epsilon} \bigr] A_{\epsilon}(m, d) N^{\frac{m-1}{2m} + \epsilon} \\ 
&\leq \frac{1}{8} A_{\epsilon}(m, d) +  \frac{1}{8} A_{\epsilon}(m, d)  N^{\frac{m-1}{2m} + \epsilon} \leq \frac{1}{4} A_{\epsilon}(m, d)  N^{\frac{m-1}{2m} + \epsilon}. 
\end{align*}  
At the penultimate step above, we have first chosen $D$ large enough to satisfy 
\begin{equation}  2 R_{n, d}^{\frac{1}{2}} C_m^{\frac{m-1}{2m}+\epsilon} D^{-m \epsilon} < \frac{1}{8}, \label{D2} 
\end{equation} 
and then chosen $A_{\epsilon}(m, d)$ sufficiently large so that 
\[  A_{\epsilon}(m, d) > 8 R_{n, d} D^m.  \]
This completes the proof of \eqref{aim-to-show-eq:goal3}.
\vskip0.1in
\noindent Finally, we turn to the proof of \eqref{aim-to-show-eq:goal4}, which specifies the contribution from $\Omega^{\ast \ast}(V_m) $. This set is a finite subset of cardinality at most $N$ of $V_m(\mathbb R) \cap Z(P) = \bigl[ V_m \cap Z_{\mathbb C}(P)\bigr] \cap \bigl[ \mathbb R^n + i \{ \vec{0} \} \bigr]$. The choice of the partitioning polynomial $P$ from Theorem \ref{thm:polyvar} ensures that $P$ does not vanish identically on any irreducible component of $V_m$, hence by Lemma \ref{lem:dim-intersection}, 
\[ \text{dim}(V_m \cap Z_{\mathbb C}(P)) < \text{dim}(V_m) = m.  \]
This sets the stage for induction based on the dimension $m$. However, we also need a bound on the degree of $V_m \cap Z_{\mathbb C}(P)$, in order to keep track of the implicit constants. By the generalized Bezout's theorem (Lemma \ref{gen-Bezout} and \eqref{eqn:Bezout}), 
\[ \deg (V_m \cap Z_{\mathbb C}(P)) \leq \deg (V_m) \deg(P) \leq \deg(V) \deg(P) \leq d D. \] 

The induction hypothesis with $(m', N') = (m-1, N)$ yields 
\begin{align*} 
\bigl|\bigl| H_{\Omega^{\ast \ast}(V_m)}\bigr| \bigr|_{2 \rightarrow 2} &\leq \mathfrak C_{\text{alg}}(N; m-1, n, dD) \\ 
&\leq \begin{cases}  dD &\text{ if  } m = 1, \\ A_{\epsilon}(m-1, dD) N^{\frac{m-2}{2(m-1)} + \epsilon} &\text{ if } 2 \leq m \leq n, \end{cases} \\ 
&\leq \frac{1}{8} A_{\epsilon}(m,  d) N^{\frac{m-1}{2m} + \epsilon}. 
\end{align*}     
The last step follows by choosing $A_{\epsilon}(m,  d)$ large enough, namely 
\[ \frac{A_{\epsilon}(m, d)}{8} > \begin{cases} dD &\text{ if  } m = 1, \\ A_{\epsilon}(m-1, dD) &\text{ if } 2 \leq m \leq n, \end{cases}  \] 
recalling that $D$ depends only on $\epsilon, m,n, d$.
This completes the proof of \eqref{aim-to-show-eq:goal4}, and hence the proof of Theorem \ref{thm:higherdim}. \qed

\appendix
\section{Auxiliary lemmas} \label{sec:appendix}

\subsection{Basic properties of the operator norm of $H_{\om}$}
\begin{lem}[Invariance of operator norm under translation and dilation] \label{lem:modify} Let $\om\subset \R^n$. For $c \in (\R_+)^n$ and $w\in \R^n$, define $c\om = \{ (c_1 v_1, \cdots, c_nv_n)\in \R^n: v\in \om \}$ and $\om+w = \{ v+w : v\in \om \}$. Then 
\[ \norm{H_{\om}} \op = \norm{H_{c\om+w}} \op.\]
\end{lem}
\begin{proof}[Sketch of proof]
It suffices to show that 
\[ \norm{H_{\om}} \op = \norm{H_{c \om}} \op  \;\text{ and } \; \norm{H_{\om}} \op = \norm{H_{\om+w}} \op.\] We only prove the second equality, leaving the verification of the first one to the interested reader. For a given $f\in L^2(\R^{n+1})$, let $g(y',y_{n+1}) = f(y'+w y_{n+1}, y_{n+1})$. Observe that 
\[ [H_{v+w} f](x',x_{n+1}) = [H_{v} g](x'-wx_{n+1},x_{n+1}).\]
By taking sup over $v\in \om$ and then $L^2$ norm, we find that $\norm{H_{\om+w} f}_{L^2} = \norm{H_{\om} g}_{L^2} $. Since $\norm{f}_{L^2} = \norm{g}_{L^2}$, this implies $\norm{H_{\om +w}}\op \leq \norm{H_{\om}} \op$. The reverse inequality can be shown similarly.
\end{proof}

\begin{lem} \label{lem:slice} For $\om \subset \R^n$ and $w \in \R^l$, let ${\om}_w = \om \times \{ w \}$. Then
\[ \norm{H_\om} \opn{n+1} = \norm{H_{\om_w}} \opn{n+l+1}. \]
\end{lem}
\begin{proof}
Since $\om_w = \om \times \{0\} + \{0\} \times \{ w \}$, by Lemma \ref{lem:modify} we may assume that $w=0\in \R^l$.

For a given $v\in V$, we write $\tilde{v} = (v,0,\ldots,0) \in \om_0$. For $g \in L^2(\R^{n+l+1})$, there is the identity
\[ H_{\tilde{v}} g(x) = H_{v} [g_{x_{n+1},\ldots,x_{n+l}}] (x_1,\ldots,x_n,x_{n+l+1}), \]
where we write $g(x) = g_{x_{n+1},\ldots,x_{n+l}} (x_1,\ldots,x_n,x_{n+l+1}  )$. This yields \[ \norm{H_{\om_0}} \opn{n+l+1} \leq \norm{H_\om} \opn{n+1}. \]

For the opposite inequality, let $f \in L^2(\R^{n+1})$. Define 
\[ \tilde{f}(x) = f(x_1,\ldots,x_n,x_{n+l+1}) \chi(x_{n+1},\ldots,x_{n+l}) \]
for a fixed function $\norm{\chi}_{L^2(\R^l)} = 1$. Observe that 
\[ H_{\tilde{v}} \tilde{f}(x) = \chi(x_{n+1},\ldots,x_{n+l})H_{v} f (x_1,\ldots,x_n,x_{n+l+1}) . \] 
This yields the reverse inequality.
\end{proof}

\subsection{Algebraic facts needed in Section \ref{proof:thm3d:section}}
In the proof of Proposition \ref{lem:3d} in Section \ref{proof:thm3d:section}, we appealed to a structure theorem for bivariate polynomials. The goal of this section is to prove this result, which has been stated in Lemma \ref{lem:decomposition} below. The proof relies on an estimate due to Basu, Pollack and Roy \cite{Basu-Pollack-Roy} on the number of ``cells'' or connected components generated by the zero set of a family of polynomials in an algebraic variety. This result has been refined further in subsequent work \cite{BB}, but the following version suffices for our purposes.     
\begin{theorem} \cite[Theorem 1]{Basu-Pollack-Roy} \label{thm:BPR}
Let $W \subseteq \mathbb R^n$ be an algebraic variety of real dimension $m$, defined as the common zero set of real polynomials of degree at most $d$. Let $\mathcal Q = \{ Q_1, \cdots, Q_s \}$ be a family of real polynomials of degree at most $d$. Then the total number of (semi-algebraically) connected components of $W \setminus  Z_{\mathbb R}(Q_1, \cdots, Q_s)$ is at most $O(d^n)$, where the implicit constant may depend only on $s,n$. 
\end{theorem} 
\begin{lem}\label{lem:decomposition}
Let $P$ be a bivariate polynomial of degree $d$ that is not identically zero. Then, possibly after an affine change of coordinates, we may write $Z(P) = Z_{\mathbb R}(P) = \{(x, y) \in \mathbb R^2 : P(x, y) = 0 \}$ as the disjoint union of $O(d^2)$ points and $O(d^2)$ curves, where each curve is given by a graph of the form $\{(x,g(x)): x\in I \}$ for some continuous function $g$ and some interval $I\subset \R$. The constant implicit in the big oh notation $O$ is absolute. 
\end{lem}
\begin{proof}
Without loss of generality, by an affine linear transformation if necessary, we may take a bivariate polynomial $P$ of degree $d$ to be of the form \begin{equation}  P(x, y) = y^d + \sum_{j=1}^{d} a_j(x) y^{d-j} \text{ where $a_j$ are univariate polynomials.} \label{monic-form} \end{equation} We may also assume that $P$ is square-free, since the presence of repeated factors leaves $Z(P)$ invariant; in other words, $P$ admits a unique factorization into distinct, irreducible polynomials, $P = P_1 \cdots P_m$, where each $P_j$ is of the form \eqref{monic-form}. Irreducibility implies that (a) $P_j$ and $\partial_y P_j$ do not share a common factor for any $j$, and (b) $P_j$ and $P_k$ do not share a common factor for any choice of $j \ne k$. It follows then (by induction on $m$ for example) that $P$ and $\partial_y P$ do not share a common factor either. Here $\partial_y P = \frac{\partial P}{\partial y}$.
\vskip0.1in
\noindent We now decompose
\[ Z(P) = \bigl[ Z(P, \partial_y P)\bigr] \bigsqcup \bigl[Z(P)\setminus Z(\partial_y P)]. \]
We have shown in the previous paragraph that the polynomials $P$ and $\partial_y P$ have no common factors, hence by B\'{e}zout's theorem \cite[Theorem 2.7]{Sheffer}, we know that $Z(P, \partial_y P)$ is a finite set of cardinality at most $\deg(P) \deg(\partial_y P)  \leq d(d-1) = O(d^2)$. On the other hand, by the implicit function theorem, each connected component of the remainder $Z(P)\setminus Z(\partial_y P)$ can be expressed as a graph of the form $\{ (x,g(x)): x\in I \}$ for some function $g$ and some interval $I\subset \R$. By Theorem \ref{thm:BPR} with $W = Z(P)$, $\mathcal Q = \{ \partial_y P\}$, $n=2$ and $s = 1$, we know that $Z(P)\setminus Z(\partial_y P)$ has $O(d^2)$ connected components. This completes the proof. 
\end{proof} 
\subsection{Algebraic facts needed in Section \ref{proof-thm-highd-section}}
Let us recall that for $u \in \mathbb S^n$, $P_u: \mathbb R^n \rightarrow \mathbb R$ denotes the function $P_{u}(y) = u\cdot \inn{y,1}$. In the proof of Proposition \ref{lem:boundE}, we made use of the following lemma. 
\begin{lem}  \label{lem:dim} Let $V$ be a variety in $\C^n$ of dimension $\geq 1$. Then
\[ \dim (Z_{\C}(P_u)\cap V) < \dim V \] for almost every $u\in \mathbb S^n$.
\end{lem}
\begin{proof}
Since any variety $V$ is the unique and disjoint union of irreducible components, we may assume that $V$ is irreducible. When $\dim V = n$, we know that $V=\C^n$ and therefore $\dim (Z_{\C}(P_u) \cap V) = \dim Z_{\C}(P_u) = n-1 < \dim V$. 
\vskip0.1in
\noindent Suppose now that $1\leq \dim V < n$. If $\dim(Z_{\C}(P_u) \cap V) \geq \dim V$, then $Z_{\C}(P_u) \cap V = V$ by the definition of dimension. Therefore, it suffices to show that the Lebesgue surface area measure of the set  
\[ \mathbb S^n_V := \{u\in \mathbb S^n : V \subset Z_{\C}(P_u) \} \]
is 0. Fix a point $z=(z_1,\cdots,z_n) \in V \subset \C^n$ and let $x=(x_1,\cdots,x_n)\in \R^n$, where each $x_j = \text{Re}(z_j)$. Observe that $\mathbb S^n_V \subset \mathbb S^n_{\{z\}} \subset \mathbb S^n_{\{x\}}$ and that 
\[ \mathbb S^n_{\{x\}} = \{u\in \mathbb S^n : u \cdot \inn{x,1}  = 0 \} \]
is of Lebesgue measure 0 since it is the intersection of $\mathbb S^n$ with a hyperplane in $\R^{n+1}$ through the origin. Thus, $\mathbb S^n_V$ has measure 0.
\end{proof}

\nocite{Car, Str}

\bibliographystyle{amsplain}

\end{document}